\documentclass[platex]{amsart}
\usepackage{amssymb,amsmath,amsthm,empheq,comment,color}
\usepackage[utf8]{inputenc}
\usepackage{enumerate}

\usepackage[top=30truemm,bottom=30truemm,left=20truemm,right=20truemm]{geometry}

\theoremstyle{definition}
\newtheorem{theorem}{Theorem}[section]
\newtheorem{proposition}[theorem]{Proposition}

\newtheorem{lemma}[theorem]{Lemma}
\newtheorem{corollary}[theorem]{Corollary}

\newtheorem*{remark*}{Remark}

\DeclareMathOperator{\re}{Re}
\DeclareMathOperator{\im}{Im}

\DeclareMathOperator{\Mod}{Mod}

\DeclareMathOperator{\app}{app}

\title[Minimal mass blow-up solutions for NLS with a singlar potential]{Minimal mass blow-up solutions for nonlinear Schr\"{o}dinger equations with a singular potential}
\author[N. Matsui]{Naoki Matsui}
\date{\today}
\address[N. Mastui]{Department of Mathematics\\ Tokyo University of Science\\ 1-3 Kagurazaka, Shinjuku-ku, Tokyo 162-8601, Japan}
\email[N. Matsui]{1120703@ed.tus.ac.jp}


\begin{document}
\maketitle

\begin{abstract}
We consider the following nonlinear Schr\"{o}dinger equation with an inverse potential:
\[
i\frac{\partial u}{\partial t}+\Delta u+|u|^{\frac{4}{N}}u\pm\frac{1}{|x|^{2\sigma}}\log|x|u=0
\]
in $\mathbb{R}^N$. From the classical argument, the solution with subcritical mass ($\|u\|_2<\|Q\|_2$) is global and bounded in $H^1(\mathbb{R}^N)$. Here, $Q$ is the ground state of the mass-critical problem. Therefore, we are interested in the existence and behaviour of blow-up solutions for the threshold ($\left\|u_0\right\|_2=\left\|Q\right\|_2$).
\end{abstract}

\section{Introduction}
We consider the following nonlinear Schr\"{o}dinger equation with an inverse-log potential:
\begin{align}
\label{NLS}
i\frac{\partial u}{\partial t}+\Delta u+|u|^{\frac{4}{N}}u\pm\frac{1}{|x|^{2\sigma}}\log|x|u=0,\quad (t,x)\in\mathbb{R}\times\mathbb{R}^N,
\end{align}
where $N\in\mathbb{N}$ and $\sigma\in\mathbb{R}$. It is well known that if
\begin{align}
\label{index1}
0<\sigma<\min\left\{\frac{N}{2},1\right\},
\end{align}
then \eqref{NLS} is locally well-posed in $H^1(\mathbb{R}^N)$ from \cite[Proposition 3.2.2, Proposition 3.2.5, Theorem 3.3.9, and Proposition 4.2.3]{CSSE}. This means that for any initial value $u_0\in H^1(\mathbb{R}^N)$, there exists a unique maximal solution $u\in C((T_*,T^*),H^1(\mathbb{R}^N))\cap C^1((T_*,T^*),H^{-1}(\mathbb{R}^N))$ for \eqref{NLS} with $u(0)=u_0$. Moreover, the mass (i.e., $L^2$-norm) and energy $E$ of the solution $u$  are conserved by the flow, where 
\[
E(u):=\frac{1}{2}\left\|\nabla u\right\|_2^2-\frac{1}{2+\frac{4}{N}}\left\|u\right\|_{2+\frac{4}{N}}^{2+\frac{4}{N}}\mp\frac{1}{2}\int_{\mathbb{R}^N}\frac{1}{|x|^{2\sigma}}\log|x||u(x)|^2dx.
\]
Furthermore, the blow-up alternative holds:
\[
T^*<\infty\quad \mbox{implies}\quad \lim_{t\nearrow T^*}\left\|\nabla u(t)\right\|_2=\infty.
\]

We define $\Sigma^k$ by
\[
\Sigma^k:=\left\{u\in H^k\left(\mathbb{R}^N\right)\ \middle|\ |x|^ku\in L^2\left(\mathbb{R}^N\right)\right\},\quad \|u\|_{\Sigma^k}^2:=\|u\|_{H^k}^2+\||x|^ku\|_2^2.
\]
Particularly, $\Sigma^1$ is called the virial space. If $u_0\in \Sigma^1$, then the solution $u$ for \eqref{NLS} with $u(0)=u_0$ belongs to $C((T_*,T^*),\Sigma^1)$ from \cite[Lemma 6.5.2]{CSSE}.

Moreover, we consider the case
\begin{align}
\label{index2}
0<\sigma<\min\left\{\frac{N}{4},1\right\}.
\end{align}
If $u_0\in H^2(\mathbb{R}^N)$, then the solution $u$ for \eqref{NLS} with $u(0)=u_0$ belongs to $C((T_*,T^*),H^2(\mathbb{R}^N))\cap C^1((T_*,T^*),L^2(\mathbb{R}^N))$ and $|x|\nabla u\in C((T_*,T^*),L^2(\mathbb{R}^N))$ from \cite[Theorem 5.3.1]{CSSE}. Furthermore, if $u_0\in \Sigma^2$, then the solution $u$ for \eqref{NLS} with $u(0)=u_0$ belongs to $C((T_*,T^*),\Sigma^2)\cap C^1((T_*,T^*),L^2(\mathbb{R}^N))$ and $|x|\nabla u\in C((T_*,T^*),L^2(\mathbb{R}^N))$ from the same proof as in \cite[Lemma 6.5.2]{CSSE}.

\subsection{Critical problem}
Firstly, we describe the results regarding the mass-critical problem:
\begin{align}
\label{CNLS}
i\frac{\partial u}{\partial t}+\Delta u+|u|^{\frac{4}{N}}u=0,\quad (t,x)\in\mathbb{R}\times\mathbb{R}^N,
\end{align}
In particular, \eqref{NLS} with $\sigma=0$ is reduced to \eqref{CNLS}.

It is well known (\cite{BLGS,KGS,WGS}) that there exists a unique classical solution $Q$ for
\[
-\Delta Q+Q-\left|Q\right|^{\frac{4}{N}}Q=0,\quad Q\in H^1(\mathbb{R}^N),\quad Q>0,\quad Q\mathrm{\ is\ radial},
\]
which is called the ground state. If $\|u\|_2=\|Q\|_2$ ($\|u\|_2<\|Q\|_2$, $\|u\|_2>\|Q\|_2$), we say that $u$ has the \textit{critical mass} (\textit{subcritical mass}, \textit{supercritical mass}, respectively).

We note that $E_{\mathrm{crit}}(Q)=0$, where $E_{\mathrm{crit}}$ is the energy with respect to \eqref{CNLS}. Moreover, the ground state $Q$ attains the best constant in the Gagliardo-Nirenberg inequality
\[
\left\|v\right\|_{2+\frac{4}{N}}^{2+\frac{4}{N}}\leq\left(1+\frac{2}{N}\right)\left(\frac{\left\|v\right\|_2}{\left\|Q\right\|_2}\right)^{\frac{4}{N}}\left\|\nabla v\right\|_2^2\quad\mbox{for }v\in H^1(\mathbb{R}^N).
\]
Therefore, for all $v\in H^1(\mathbb{R}^N)$,
\[
E_{\mathrm{crit}}(v)\geq \frac{1}{2}\left\|\nabla v\right\|_2^2\left(1-\left(\frac{\left\|v\right\|_2}{\left\|Q\right\|_2}\right)^{\frac{4}{N}}\right)
\]
holds. This inequality and the mass and energy conservations imply that any subcritical mass solution for \eqref{CNLS} is global and bounded in $H^1(\mathbb{R}^N)$.

Regarding the critical mass case, we apply the pseudo-conformal transformation
\[
u(t,x)\ \mapsto\ \frac{1}{\left|t\right|^\frac{N}{2}}u\left(-\frac{1}{t},\pm\frac{x}{t}\right)e^{i\frac{\left|x\right|^2}{4t}}
\]
to the solitary wave solution $u(t,x):=Q(x)e^{it}$. Then we obtain
\[
S(t,x):=\frac{1}{\left|t\right|^\frac{N}{2}}Q\left(\frac{x}{t}\right)e^{-\frac{i}{t}}e^{i\frac{\left|x\right|^2}{4t}},
\]
which is also a solution for \eqref{CNLS} and satisfies
\[
\left\|S(t)\right\|_2=\left\|Q\right\|_2,\quad \left\|\nabla S(t)\right\|_2\sim\frac{1}{\left|t\right|}\quad (t\nearrow 0).
\]
Namely, $S$ is a minimal mass blow-up solution for \eqref{CNLS}. Moreover, $S$ is the only finite time blow-up solution for \eqref{CNLS} with critical mass, up to the symmetries of the flow (see \cite{MMMB}).

Regarding the supercritical mass case, there exists a solution $u$ for \eqref{CNLS} such that
\[
\left\|\nabla u(t)\right\|_2\sim\sqrt{\frac{\log\bigl|\log\left|T^*-t\right|\bigr|}{T^*-t}}\quad (t\nearrow T^*)
\]
(see \cite{MRUPB,MRUDB}).

\subsection{Previous results}
Le Coz, Martel, and Rapha\"{e}l \cite{LMR} based on the methodology of \cite{RSEU} obtains the following results for
\begin{align}
\label{DPNLS}
i\frac{\partial u}{\partial t}+\Delta u+|u|^{\frac{4}{N}}u\pm|u|^{p-1}u=0,\quad (t,x)\in\mathbb{R}\times\mathbb{R}^N.
\end{align}

\begin{theorem}[\cite{LMR,NDP}]
\label{theorem:LMR}
Let $1<p<1+\frac{4}{N}$, and $\pm=+$. Then for any energy level $E_0\in\mathbb{R}$, there exist $t_0<0$ and a radially symmetric initial value $u_0\in H^1(\mathbb{R}^N)$ with
\[
\|u_0\|_2=\|Q\|_2,\quad E(u_0)=E_0
\]
such that the corresponding solution $u$ for \eqref{DPNLS} with $u(t_0)=u_0$ blows up at $t=0$ with a blow-up rate of
\[
\|\nabla u(t)\|_2=\frac{C(p)+o_{t\nearrow 0}(t)}{|t|^{\sigma}},
\]
where $\sigma=\frac{4}{4+N(p-1)}$ and $C(p)>0$.
\end{theorem}

\begin{theorem}[\cite{LMR}]
\label{theorem:LMR2}
Let $1<p<1+\frac{4}{N}$, and $\pm=-$. If an initial value has critical mass, then the corresponding solution for \eqref{DPNLS} with $u(0)=u_0$ is global and bounded in $H^1(\mathbb{R}^N)$.
\end{theorem}

\cite{NI} obtains the  following results for
\begin{align}
\label{IPNLS}
i\frac{\partial u}{\partial t}+\Delta u+|u|^{\frac{4}{N}}u\pm\frac{1}{|x|^{2\sigma}}u=0,\quad (t,x)\in\mathbb{R}\times\mathbb{R}^N.
\end{align}

\begin{theorem}[\cite{NI}]
\label{theorem:NI1}
Assume \eqref{index2}. Then for any energy level $E_0\in\mathbb{R}$, there exist $t_0<0$ and a radially symmetric initial value $u_0\in H^1(\mathbb{R}^N)$ with
\[
\|u_0\|_2=\|Q\|_2,\quad E(u_0)=E_0
\]
such that the corresponding solution $u$ for \eqref{IPNLS} with $\pm=+$ and $u(t_0)=u_0$ blows up at $t=0$. Moreover,
\[
\left\|u(t)-\frac{1}{\lambda(t)^\frac{N}{2}}P\left(t,\frac{x}{\lambda(t)}\right)e^{-i\frac{b(t)}{4}\frac{|x|^2}{\lambda(t)^2}+i\gamma(t)}\right\|_{\Sigma^1}\rightarrow 0\quad (t\nearrow 0)
\]
holds for some blow-up profile $P$ and $C^1$ functions $\lambda:(t_0,0)\rightarrow(0,\infty)$ and $b,\gamma:(t_0,0)\rightarrow\mathbb{R}$ such that
\begin{align*}
P(t)&\rightarrow Q\quad\mbox{in}\ H^1(\mathbb{R}^N),&\lambda(t)&=C_1(\sigma)|t|^{\frac{1}{1+\sigma}}\left(1+o(1)\right),\\
b(t)&=C_2(\sigma)|t|^{\frac{1-\sigma}{1+\sigma}}\left(1+o(1)\right),& \gamma(t)^{-1}&=O\left(|t|^{\frac{1-\sigma}{1+\sigma}}\right)
\end{align*}
as $t\nearrow 0$.
\end{theorem}

On the other hand, the following result holds in \eqref{IPNLS} with $\pm=-$.

\begin{theorem}[\cite{NI}]
\label{theorem:NI2}
Assume $N\geq 2$ and \eqref{index1}. If $u_0\in H^1_\mathrm{rad}(\mathbb{R}^N)$ such that $\|u_0\|_2=\|Q\|_2$, the corresponding solution $u$ for \eqref{IPNLS} with $\pm=-$ and $u(0)=u_0$ is global and bounded in $H^1(\mathbb{R}^N)$.
\end{theorem}

\subsection{Main results}
It is immediately clear from the classical argument that all subcritical-mass solutions for \eqref{NLS} with \eqref{index1} are global and bounded in $H^1(\mathbb{R}^N)$.

In contrast, regarding critical mass in \eqref{NLS} with $\pm=-$, we obtain the following result:

\begin{theorem}[Existence of a minimal-mass blow-up solution]
\label{theorem:EMBS}
Assume \eqref{index2}. Then for any energy level $E_0\in\mathbb{R}$, there exist $t_0<0$ and a radially symmetric initial value $u_0\in H^1(\mathbb{R}^N)$ with
\[
\|u_0\|_2=\|Q\|_2,\quad E(u_0)=E_0
\]
such that the corresponding solution $u$ for \eqref{NLS} with $\pm=-$ and $u(t_0)=u_0$ blows up at $t=0$. Moreover,
\[
\left\|u(t)-\frac{1}{\lambda(t)^\frac{N}{2}}P\left(t,\frac{x}{\lambda(t)}\right)e^{-i\frac{b(t)}{4}\frac{|x|^2}{\lambda(t)^2}+i\gamma(t)}\right\|_{\Sigma^1}\rightarrow 0\quad (t\nearrow 0)
\]
holds for some blow-up profile $P$ and $C^1$ functions $\lambda:(t_0,0)\rightarrow(0,\infty)$ and $b,\gamma:(t_0,0)\rightarrow\mathbb{R}$ such that
\begin{align*}
P(t)&\rightarrow Q\quad\mbox{in}\ H^1(\mathbb{R}^N),&\lambda(t)&\approx|t|^{\frac{1}{1+\sigma}}\left|\log|t|\right|^{\frac{1}{2+2\sigma}},\\
b(t)&\approx|t|^{\frac{1-\sigma}{1+\sigma}}|\log|t||^{\frac{1}{1+\sigma}},& \gamma(t)^{-1}&=O\left(|t|^{\frac{1-\sigma}{1+\sigma}}\right)
\end{align*}
as $t\nearrow 0$.
\end{theorem}

On the other hand, the following result holds in \eqref{NLS} with $\pm=+$.

\begin{theorem}[Non-existence of a radial minimal-mass blow-up solution]
\label{theorem:NEMBS}
Assume $N\geq 2$ and \eqref{index1}. If $u_0\in H^1_\mathrm{rad}(\mathbb{R}^N)$ such that $\|u_0\|_2=\|Q\|_2$, the corresponding solution $u$ for \eqref{NLS} with $\pm=+$ and $u(0)=u_0$ is global and bounded in $H^1(\mathbb{R}^N)$.
\end{theorem}

This result is proved in the same way as for Theorem \ref{theorem:LMR2} and Theorem \ref{theorem:NI2}.

\subsection{Comments regarding the main results}
We present some comments regarding Theorem \ref{theorem:EMBS}.

The assumption $\sigma>0$ is used only thet \eqref{NLS} is locally well-posedness in $H^1(\mathbb{R}^N)$. Asuming \eqref{NLS} is locally well-posedness in $\Sigma$, it can be proven in $\sigma=0$ since the proof is done in $\Sigma$.

Near the origin,
\[
\frac{1}{|x|^{2\sigma}}\leq -\frac{1}{|x|^{2\sigma}}\log|x|\leq\frac{1}{|x|^{2(\sigma+\epsilon)}}
\]
holds for any $\epsilon>0$. The corresponding blow-up rates from Theorem \ref{theorem:NI1} and Theorem \ref{theorem:EMBS} satisfy
\[
|t|^{-\frac{1}{1+\sigma}}\gtrsim |t|^{-\frac{1}{1+\sigma}}|\log|t||^{-\frac{1}{2+2\sigma}} \gtrsim |t|^{-\frac{1}{1+\sigma+\epsilon}}.
\]
This suggests that a large or small relationship between the strength of the potential's singularities gives a large or small relationship for the blow-up rates.

This result suggests that it is possible to construct a critical-mass blow-up solution if an equation such as
\[
0=i\frac{\partial v}{\partial s}+\Delta v-v+f(v)+g_1(\lambda)g_2(y,v)+\mbox{error terms}
\]
can be obtained by separating $\lambda$ from $v$, as in \eqref{veq}. In order to construct an blow-up solution, it is necessary to be able to obtain at least $\lambda_{\app},b_{\app}$ in Lemma \ref{lbsol} and define $\mathcal{F}$ in Lemma \ref{paraini}. Furthermore, in the case of $g(\lambda)=O(\lambda^2)$, we can expect that there exists a critical-mass blow-up solution such that the blow-up rate is $t^{-1}$.

\section{Notations}
In this section, we introduce the notation used in this paper.

Let
\[
\mathbb{N}:=\mathbb{Z}_{\geq 1},\quad\mathbb{N}_0:=\mathbb{Z}_{\geq 0}.
\]
We define
\begin{align*}
(u,v)_2&:=\re\int_{\mathbb{R}^N}u(x)\overline{v}(x)dx,&\left\|u\right\|_p&:=\left(\int_{\mathbb{R}^N}|u(x)|^pdx\right)^\frac{1}{p},\\
f(z)&:=|z|^\frac{4}{N}z,&F(z)&:=\frac{1}{2+\frac{4}{N}}|z|^{2+\frac{4}{N}}\quad \mbox{for $z\in\mathbb{C}$}.
\end{align*}
By identifying $\mathbb{C}$ with $\mathbb{R}^2$, we denote the differentials of $f$ and $F$ by $df$ and $dF$, respectively. We define
\[
\Lambda:=\frac{N}{2}+x\cdot\nabla,\quad L_+:=-\Delta+1-\left(1+\frac{4}{N}\right)Q^\frac{4}{N},\quad L_-:=-\Delta+1-Q^\frac{4}{N}.
\]
Namely, $\Lambda$ is the generator of $L^2$-scaling, and $L_+$ and $L_-$ come from the linearised Schr\"{o}dinger operator to close $Q$. Then
\[
L_-Q=0,\quad L_+\Lambda Q=-2Q,\quad L_-|x|^2Q=-4\Lambda Q,\quad L_+\rho=|x|^2 Q,\quad L_-xQ=-\nabla Q
\]
hold, where $\rho\in\mathcal{S}(\mathbb{R}^N)$ is the unique radial solution for $L_+\rho=|x|^2 Q$. Note that there exist $C_\alpha,\kappa_\alpha>0$ such that
\[
\left|\left(\frac{\partial}{\partial x}\right)^\alpha Q(x)\right|\leq C_\alpha Q(x),\quad \left|\left(\frac{\partial}{\partial x}\right)^\alpha \rho(x)\right|\leq C_\alpha(1+|x|)^{\kappa_\alpha} Q(x).
\]
for any multi-index $\alpha$. Furthermore, there exists $\mu>0$ such that for all $u\in H_{\mathrm{rad}}^1(\mathbb{R}^N)$,
\begin{align}
\label{Lcoer}
&\left\langle L_+\re u,\re u\right\rangle+\left\langle L_-\im u,\im u\right\rangle\nonumber\\
\geq&\ \mu\left\|u\right\|_{H^1}^2-\frac{1}{\mu}\left({(\re u,Q)_2}^2+{(\re u,|x|^2 Q)_2}^2+{(\im u,\rho)_2}^2\right)
\end{align}
(e.g., see \cite{MRO,MRUPB,RSEU,WL}). We denote by $\mathcal{Y}$ the set of functions $g\in C^{\infty}(\mathbb{R}^N\setminus\{0\})\cap C(\mathbb{R}^N)\cap H^1_{\mathrm{rad}}(\mathbb{R}^N)$ such that
\[
\exists C_\alpha,\kappa_\alpha>0,\ |x|\geq 1\Rightarrow \left|\left(\frac{\partial}{\partial x}\right)^\alpha g(x)\right|\leq C_\alpha(1+|x|)^{\kappa_{\alpha}}Q(x)
\]
for any multi-index $\alpha$. Moreover, we defined by $\mathcal{Y}'$ the set of functions $g\in\mathcal{Y}$ such that
\[
\Lambda g\in H^1(\mathbb{R}^N)\cap C(\mathbb{R}^N).
\]

Finally, we use the notation $\lesssim$ and $\gtrsim$ when the inequalities hold up to a positive constant. We also use the notation $\approx$ when $\lesssim$ and $\gtrsim$ hold. Moreover, positive constants $C$ and $\epsilon$ are sufficiently large and small, respectively.

\section{Construction of a blow-up profile}
\label{sec:constprof}
In this section, we construct a blow-up profile $P$ and introduce a decomposition of functions based on the methodology in \cite{LMR,RSEU}.

Heuristically, we state the strategy. We look for a blow-up solution in the form of \eqref{transutov}:
\[
u(t,x)=\frac{1}{\lambda(s)^\frac{N}{2}}v\left(s,y\right)e^{-i\frac{b(s)|y|^2}{4}+i\gamma(s)},\quad y=\frac{x}{\lambda(s)},\quad \frac{ds}{dt}=\frac{1}{\lambda(s)^2},
\]
where $v$ satisfies
\begin{align}
\label{veq}
0&=i\frac{\partial v}{\partial s}+\Delta v-v+f(v)-\lambda^\alpha\log\lambda \frac{1}{|y|^{2\sigma}}v-\lambda^\alpha\frac{1}{|y|^{2\sigma}}\log|y|v\nonumber\\
&\hspace{20pt}-i\left(\frac{1}{\lambda}\frac{\partial \lambda}{\partial s}+b\right)\Lambda v+\left(1-\frac{\partial \gamma}{\partial s}\right)v+\left(\frac{\partial b}{\partial s}+b^2\right)\frac{|y|^2}{4}v-\left(\frac{1}{\lambda}\frac{\partial \lambda}{\partial s}+b\right)b\frac{|y|^2}{2}v,
\end{align}
where $\alpha=2-2\sigma$. Since we look for a blow-up solution, it may holds that $\lambda(s)\rightarrow 0$ as $s\rightarrow\infty$. Therefore, it seems that $\lambda^\alpha |y|^{-2\sigma}v$ is ignored. By ignoring $\lambda^\alpha |y|^{-2\sigma}v$,
\[
v(s,y)=Q(y),\quad \frac{1}{\lambda}\frac{\partial \lambda}{\partial s}+b=1-\frac{\partial \gamma}{\partial s}=\frac{\partial b}{\partial s}+b^2=0
\]
is a solution of \eqref{veq}. Accordingly, $v$ is expected to be close to $Q$. We now consider the case where $\sigma=0$, i.e., the critical problem. Then $\lambda^2v$ corresponds to the linear term with the constant coefficient and can be removed by an appropriate transformation. In other words, $\lambda^2v$ is a negligible term for the construction of minimal-mass blow-up solutions. This suggests that $\alpha=2$ may be the threshold for ignoring the term in the context of minimal-mass blow-up. Therefore, $\lambda^\alpha |y|^{-2\sigma}v$ may become a non-negligible term if $\alpha<2$, i.e., $\sigma>0$. Also, \eqref{veq} is difficult to solve explicitly. Consequently, we construct an approximate solution $P$ that is close to $Q$ and fully incorporates the effects of $\lambda^\alpha |y|^{-2\sigma}v$, e.g., the singularity of the origin.

For $K\in\mathbb{N}$, we define
\[
\Sigma_K:=\left\{\ (j,k_1,k_2)\in{\mathbb{N}_0}^3\ \middle|\ j+k_1+k_2\leq K\ \right\}.
\]

\begin{proposition}
\label{theorem:constprof}
Let $K,K'\in\mathbb{N}$ be sufficiently large. Let $\lambda(s)>0$ and $b(s)\in\mathbb{R}$ be $C^1$ functions of $s$ such that $\lambda(s)+|b(s)|\ll 1$.

(i) \textit{Existence of blow-up profile.} For any $(j,k_1,k_2)\in\Sigma_{K+K'}$, there exist $P_{1,j,k_1,k_2}^+,P_{2,j,k_1,k_2}^+,P_{1,j,k_1,k_2}^-,P_{2,j,k_1,k_2}^-\in\mathcal{Y}'$, $\beta_{1,j,k_1,k_2},\beta_{2,j,k_1,k_2}\in\mathbb{R}$, and $\Psi\in H^1(\mathbb{R}^N)$ such that $P$ satisfies
\[
i\frac{\partial P}{\partial s}+\Delta P-P+f(P)-\lambda^\alpha\log\lambda \frac{1}{|y|^{2\sigma}}P-\lambda^\alpha \frac{1}{|y|^{2\sigma}}\log|y|P+\theta\frac{|y|^2}{4}P=\Psi,
\]
where $\alpha=2-2\sigma$, and  $P$ and $\theta$ are defined by
\begin{align*}
P(s,y)&:=Q(y)+\sum_{(j,k_1,k_2)\in\Sigma_{K+K'}}b(s)^{2j}\left(\lambda(s)^\alpha\log\lambda(s)\right)^{k_1}\lambda(s)^{k_2\alpha}\left(\lambda(s)^\alpha\log\lambda(s) P_{1,j,k_1,k_2}^+(y)+\lambda(s)^\alpha P_{2,j,k_1,k_2}^+(y)\right)\\
&\hspace{20pt}+i\sum_{(j,k_1,k_2)\in\Sigma_{K+K'}}b(s)^{2j+1}\left(\lambda(s)^\alpha\log\lambda(s)\right)^{k_1}\lambda(s)^{k_2\alpha}\left(\lambda(s)^\alpha\log\lambda(s) P_{1,j,k_1,k_2}^-(y)+\lambda(s)^\alpha P_{2,j,k_1,k_2}^-(y)\right),\\
\theta(s)&:=\sum_{(j,k_1,k_2)\in\Sigma_{K+K'}}b(s)^{2j}\left(\lambda(s)^\alpha \log\lambda(s)\right)^{k_1}\lambda(s)^{k_2\alpha}\left(-\lambda(s)^\alpha\log\lambda(s)\beta_{1,j,k_1,k_2}+\lambda(s)^\alpha\beta_{2,j,k_1,k_2}\right).
\end{align*}
Moreover, for some sufficiently small $\epsilon'>0$,
\[
\left\|e^{\epsilon'|y|}\Psi\right\|_{H^1}\lesssim\lambda^\alpha|\log\lambda|\left(\left|\frac{1}{\lambda}\frac{\partial \lambda}{\partial s}+b\right|+\left|\frac{\partial b}{\partial s}+b^2-\theta\right|\right)+(b^2+\lambda^\alpha|\log\lambda|)^{K+2}
\]
holds.

(ii) \textit{Mass and energy properties of blow-up profile.} Let define
\[
P_{\lambda,b,\gamma}(s,x):=\frac{1}{\lambda(s)^\frac{N}{2}}P\left(s,\frac{x}{\lambda(s)}\right)e^{-i\frac{b(s)}{4}\frac{|x|^2}{\lambda(s)^2}+i\gamma(s)}.
\]
Then
\begin{align*}
\left|\frac{d}{ds}\|P_{\lambda,b,\gamma}\|_2^2\right|&\lesssim\lambda^\alpha|\log\lambda|\left(\left|\frac{1}{\lambda}\frac{\partial \lambda}{\partial s}+b\right|+\left|\frac{\partial b}{\partial s}+b^2-\theta\right|\right)+(b^2+\lambda^\alpha|\log\lambda|)^{K+2},\\
\left|\frac{d}{ds}E(P_{\lambda,b,\gamma})\right|&\lesssim\frac{1}{\lambda^2}\left(\left|\frac{1}{\lambda}\frac{\partial \lambda}{\partial s}+b\right|+\left|\frac{\partial b}{\partial s}+b^2-\theta\right|+(b^2+\lambda^\alpha|\log\lambda|)^{K+2}\right)
\end{align*}
hold. Moreover,
\begin{eqnarray}
\label{Eesti}
\left|8E(P_{\lambda,b,\gamma})-\||y|Q\|_2^2\left(\frac{b^2}{\lambda^2}+\frac{2\beta_1}{2-\alpha}\lambda^{\alpha-2}\log\lambda-\beta'_1\lambda^{\alpha-2}\right)\right|\lesssim\frac{\lambda^\alpha|\log\lambda|(b^2+\lambda^\alpha|\log\lambda|)}{\lambda^2}
\end{eqnarray}
holds, where 
\[
\beta_1:=\beta_{1,0,0,0}=\frac{4\sigma\||\cdot|^{-\sigma}Q\|_2^2}{\||\cdot|Q\|_2^2}>0\quad \beta'_1:=\frac{4}{\||y|Q\|_2^2}\int_{\mathbb{R}^N}\frac{1}{|y|^{2\sigma}}\log|y|Q^2dy.
\]
\end{proposition}

\begin{proof}
See \cite{LMR,NI} for details of proofs.

We prove (i). We set
\begin{align*}
Z_1&:=\sum_{(j,k_1,k_2)\in\Sigma_{K+K'}}b^{2j}\left(\lambda^\alpha\log\lambda\right)^{k_1}\lambda^{k_2\alpha}P_{1,j,k_1,k_2}^++i\sum_{(j,k_1,k_2)\in\Sigma_{K+K'}}b^{2j+1}\left(\lambda^\alpha\log\lambda\right)^{k_1}\lambda^{k_2\alpha}P_{1,j,k_1,k_2}^-,\\
Z_2&:=\sum_{(j,k_1,k_2)\in\Sigma_{K+K'}}b^{2j}\left(\lambda^\alpha\log\lambda\right)^{k_1}\lambda^{k_2\alpha}P_{2,j,k_1,k_2}^++i\sum_{(j,k_1,k_2)\in\Sigma_{K+K'}}b^{2j+1}\left(\lambda^\alpha\log\lambda\right)^{k_1}\lambda^{k_2\alpha}P_{2,j,k_1,k_2}^-,
\end{align*}
Then $P=Q+\lambda^\alpha\log\lambda Z_1+\lambda^\alpha Z_2$ holds. Moreover, let set
\begin{align}
\label{DefTheta}
\Theta(s)&:=\sum_{(j,k)\in\Sigma_{K+K'}}b(s)^{2j}\left(\lambda(s)^\alpha\log\lambda(s)\right)^{k_1}\lambda(s)^{k_2\alpha}\left(\lambda(s)^\alpha\log\lambda(s)c^+_{1,j,k_1,k_2}+\lambda(s)^\alpha c^+_{2,j,k_1,k_2}\right),\\
\Phi&:=i\frac{\partial P}{\partial s}+\Delta P-P+f(P)-\lambda^\alpha\log\lambda\frac{1}{|y|^{2\sigma}}P-\lambda^\alpha\frac{1}{|y|^{2\sigma}}\log|y|P+\theta\frac{|y|^2}{4}P+\Theta Q,\nonumber
\end{align}
where $P_{\bullet,j,k_1,k_2}^\pm\in\mathcal{Y}'$ and $\beta_{\bullet,j,k_1,k_2},c^+_{\bullet,j,k_1,k_2}\in\mathbb{R}$ are to be determined.

As in \cite{LMR,NI}, there exist $F_{\bullet,j,k_1,k_2}^{\sigma,\pm}$, $F_{\bullet,j,k_1,k_2}^{\log,\pm}$, $F_{\bullet,j,k_1,k_2}^{\pm}$, and $\Phi$ such that 
\begin{align*}
&i\frac{\partial P}{\partial s}+\Delta P-P+f(P)-\lambda^\alpha\log\lambda\frac{1}{|y|^{2\sigma}}P-\lambda^\alpha\frac{1}{|y|^{2\sigma}}\log|y|P+\theta\frac{|y|^2}{4}P+\Theta Q\\
=&\sum_{(j,k_1,k_2)\in\Sigma_{K+K'}}b^{2j}\left(\lambda^\alpha\log\lambda\right)^{k_1+1}\lambda^{k_2\alpha}\\
&\hspace{20pt}\times\left(-L_+P_{1,j,k_1,k_2}^+-\beta_{1,j,k_1,k_2}\frac{|y|^2}{4}Q-\frac{1}{|y|^{2\sigma}}F_{1,j,k_1,k_2}^{\sigma,+}-\frac{1}{|y|^{2\sigma}}\log|y|F_{1,j,k_1,k_2}^{\log,+}+F_{1,j,k_1,k_2}^++c_{1,j,k_1,k_2}^+Q\right)\\
+&\sum_{(j,k_1,k_2)\in\Sigma_{K+K'}}b^{2j}\left(\lambda^\alpha\log\lambda\right)^{k_1}\lambda^{(k_2+1)\alpha}\\
&\hspace{20pt}\times\left(-L_+P_{2,j,k_1,k_2}^++\beta_{2,j,k_1,k_2}\frac{|y|^2}{4}Q-\frac{1}{|y|^{2\sigma}}F_{1,j,k_1,k_2}^{\sigma,+}-\frac{1}{|y|^{2\sigma}}\log|y|F_{1,j,k_1,k_2}^{\log,+}+F_{2,j,k_1,k_2}^++c_{2,j,k_1,k_2}^+Q\right)\\
+&i\sum_{(j,k_1,k_2)\in\Sigma_{K+K'}}b^{2j+1}\left(\lambda^\alpha\log\lambda\right)^{k_1+1}\lambda^{k_2\alpha}\\
&\hspace{20pt}\times\left(-L_-P_{1,j,k_1,k_2}^--\left(2j+(k_1+k_2+1)\alpha\right)P_{1,j,k_1,k_2}^+-\frac{1}{|y|^{2\sigma}}F_{1,j,k_1,k_2}^{\sigma,-}-\frac{1}{|y|^{2\sigma}}\log|y|F_{1,j,k_1,k_2}^{\log,-}+F_{1,j,k_1,k_2}^-\right)\\
+&\sum_{(j,k_1,k_2)\in\Sigma_{K+K'}}b^{2j}\left(\lambda^\alpha\log\lambda\right)^{k_1}\lambda^{(k_2+1)\alpha}\\
&\hspace{20pt}\times\left(-L_-P_{2,j,k_1,k_2}^--\left(2j+(k_1+k_2+1)\alpha\right)P_{2,j,k_1,k_2}^+-\frac{1}{|y|^{2\sigma}}F_{2,j,k_1,k_2}^{\sigma,-}-\frac{1}{|y|^{2\sigma}}\log|y|F_{2,j,k_1,k_2}^{\log,-}+F_{2,j,k_1,k_2}^-\right)\\
&\hspace{20pt}+\Phi.
\end{align*}
In particular, $F_{\bullet,j,k_1,k_2}^{\bullet,\pm}$ consists of $P_{\bullet,j',k'_1,k'_2}^{\pm}$ and $\beta_{\bullet,j',k'_1k'_2}$ for $(j',k'_1,k'_2)\in\sigma_{K+K'}$ such that $k'_2<k_2$, $k'_2\leq k'_2$ and $k'_1<k_1$, or $k'_2\leq k_2$, $k'_1\leq k_1$, and $j'<j$. Moreover,
\[
F_{1,0,0,0}^{\sigma,+}=Q,\quad F_{1,0,0,0}^{\sigma,-}=F_{2,0,0,0}^{\sigma,-}=0,\quad F_{2,0,0,0}^{\log,+}=Q,\quad F_{1,0,0,0}^{\sigma,\pm}=F_{2,0,0,0}^{\sigma,-}=0,\quad F_{\bullet,0,0,0}^\pm=0.
\]

For each $(j,k)\in\Sigma_{K+K'}$, we choose recursively $P_{j,k}^\pm\in\mathcal{Y}'$ and $\beta_{j,k},c_{j,k}^+\in\mathbb{R}$ that are solutions for the systems
\begin{empheq}[left={(S_{j,k_1,k_2})\ \empheqlbrace\ }]{align*}
&L_+P_{1,j,k_1,k_2}^++\beta_{1,j,k_1,k_2}\frac{|y|^2}{4}Q+\frac{1}{|y|^{2\sigma}}F_{1,j,k_1,k_2}^{\sigma,+}+\frac{1}{|y|^{2\sigma}}\log|y|F_{1,j,k_1,k_2}^{\log,+}-F_{1,j,k_1,k_2}^+-c_{1,j,k_1,k_2}^+Q=0,\\
&L_+P_{2,j,k_1,k_2}^+-\beta_{2,j,k_1,k_2}\frac{|y|^2}{4}Q+\frac{1}{|y|^{2\sigma}}F_{1,j,k_1,k_2}^{\sigma,+}+\frac{1}{|y|^{2\sigma}}\log|y|F_{1,j,k_1,k_2}^{\log,+}-F_{2,j,k_1,k_2}^+-c_{2,j,k_1,k_2}^+Q=0,\\
&L_-P_{1,j,k_1,k_2}^-+\left(2j+(k_1+k_2+1)\alpha\right)P_{1,j,k_1,k_2}^++\frac{1}{|y|^{2\sigma}}F_{1,j,k_1,k_2}^{\sigma,-}+\frac{1}{|y|^{2\sigma}}\log|y|F_{1,j,k_1,k_2}^{\log,-}-F_{1,j,k_1,k_2}^-=0,\\
&L_-P_{2,j,k_1,k_2}^-+\left(2j+(k_1+k_2+1)\alpha\right)P_{2,j,k_1,k_2}^++\frac{1}{|y|^{2\sigma}}F_{2,j,k_1,k_2}^{\sigma,-}+\frac{1}{|y|^{2\sigma}}\log|y|F_{2,j,k_1,k_2}^{\log,-}-F_{2,j,k_1,k_2}^-=0
\end{empheq}
and satisfy
\[
c_{\bullet,j,k_1,k_2}^+=0\ (j+k_1+k_2\leq K),\quad \frac{1}{|y|^2}P_{\bullet,0,k_1,k_2}^\pm,\frac{1}{|y|}|\nabla P_{\bullet,0,k_1,k_2}^\pm|\in L^\infty(\mathbb{R}^N)\quad(k_1+k_2=K+K').
\]
Such solutions $(P_{\bullet,j,k_1,k_2}^+,P_{\bullet,j,k_1,k_2}^-,\beta_{\bullet,j,k_1,k_2},c_{\bullet,j,k_1,k_2}^+)$ are obtained from the later Propositions \ref{theorem:Ssol}, \ref{pconti}, and \ref{Pint}.

In the same way as \cite[Proposition 2.1]{LMR}, for some sufficiently small $\epsilon'>0$, we have
\begin{align*}
\left\|e^{\epsilon'|y|}\Phi^\frac{\partial P}{\partial s}\right\|_{H^1}&\lesssim\lambda^{\alpha}|\log\lambda|\left(\left|\frac{1}{\lambda}\frac{\partial \lambda}{\partial s}+b\right|+\left|\frac{\partial b}{\partial s}+b^2-\theta\right|\right),\\
\left\|e^{\epsilon'|y|}\Phi^f\right\|_{H^1}&\lesssim\left(\lambda^\alpha|\log\lambda|\right)^{K+K'+2},\\
\left\|e^{\epsilon'|y|}\Phi^{>K+K'}\right\|_{H^1}&\lesssim\left(b^2+\lambda^{\alpha}|\log\lambda|\right)^{K+K'+2}.
\end{align*}
Moreover,
\[
\left\|e^{\epsilon'|y|}\Theta Q\right\|_{H^1}\lesssim \left(b^2+\lambda^{\alpha}|\log\lambda|\right)^{K+2}
\]
holds. Therefore, we have
\[
\left\|e^{\epsilon'|y|}\Psi\right\|_{H^1}\lesssim \lambda^{\alpha}|\log\lambda|\left(\left|\frac{1}{\lambda}\frac{\partial \lambda}{\partial s}+b\right|+\left|\frac{\partial b}{\partial s}+b^2-\theta\right|\right)+\left(b^2+\lambda^{\alpha}\right)^{K+2},
\]
where $\Psi:=\Phi-\Theta Q$.

The rest is the same as in \cite{LMR,NI}.
\end{proof}

In the rest of this section, we construct solutions $(P_{j,k}^+,P_{j,k}^-,\beta_{j,k},c_{j,k}^+)\in{\mathcal{Y}'}^2\times\mathbb{R}^2$ for systems $(S_{j,k})$ in the proof of Proposition \ref{theorem:constprof}.

\begin{proposition}
\label{theorem:Ssol}
The system $(S_{j,k_1,k_2})$ has a solution $(P_{\bullet,j,k_1,k_2}^+,P_{\bullet,j,k_1,k_2}^-,\beta_{\bullet,j,k_1,k_2},c_{\bullet,j,k_1,k_2}^+)\in\mathcal{Y}^2\times\mathbb{R}^2$.
\end{proposition}

\begin{proof}
We solve
\begin{empheq}[left={(S_{j,k_1,k_2})\ \empheqlbrace\ }]{align*}
&L_+P_{1,j,k_1,k_2}^++\beta_{1,j,k_1,k_2}\frac{|y|^2}{4}Q+\frac{1}{|y|^{2\sigma}}F_{1,j,k_1,k_2}^{\sigma,+}+\frac{1}{|y|^{2\sigma}}\log|y|F_{1,j,k_1,k_2}^{\log,+}-F_{1,j,k_1,k_2}^+-c_{1,j,k_1,k_2}^+Q=0,\\
&L_+P_{2,j,k_1,k_2}^+-\beta_{2,j,k_1,k_2}\frac{|y|^2}{4}Q+\frac{1}{|y|^{2\sigma}}F_{1,j,k_1,k_2}^{\sigma,+}+\frac{1}{|y|^{2\sigma}}\log|y|F_{1,j,k_1,k_2}^{\log,+}-F_{2,j,k_1,k_2}^+-c_{2,j,k_1,k_2}^+Q=0,\\
&L_-P_{1,j,k_1,k_2}^-+\left(2j+(k_1+k_2+1)\alpha\right)P_{1,j,k_1,k_2}^++\frac{1}{|y|^{2\sigma}}F_{1,j,k_1,k_2}^{\sigma,-}+\frac{1}{|y|^{2\sigma}}\log|y|F_{1,j,k_1,k_2}^{\log,-}-F_{1,j,k_1,k_2}^-=0,\\
&L_-P_{2,j,k_1,k_2}^-+\left(2j+(k_1+k_2+1)\alpha\right)P_{2,j,k_1,k_2}^++\frac{1}{|y|^{2\sigma}}F_{1,j,k_1,k_2}^{\sigma,-}+\frac{1}{|y|^{2\sigma}}\log|y|F_{1,j,k_1,k_2}^{\log,-}-F_{2,j,k_1,k_2}^+=0.
\end{empheq}
For $(S_{j,k_1,k_2})$, we consider the following two systems:
\begin{empheq}[left={(\tilde{S}_{j,k_1,k_2})\ \empheqlbrace\ }]{align*}
&L_+\tilde{P}_{1,j,k_1,k_2}^++\beta_{1,j,k_1,k_2}\frac{|y|^2}{4}Q+\frac{1}{|y|^{2\sigma}}F_{1,j,k_1,k_2}^{\sigma,+}+\frac{1}{|y|^{2\sigma}}\log|y|F_{1,j,k_1,k_2}^{\log,+}-F_{1,j,k_1,k_2}^+=0,\\
&L_+\tilde{P}_{2,j,k_1,k_2}^+-\beta_{2,j,k_1,k_2}\frac{|y|^2}{4}Q+\frac{1}{|y|^{2\sigma}}F_{1,j,k_1,k_2}^{\sigma,+}+\frac{1}{|y|^{2\sigma}}\log|y|F_{1,j,k_1,k_2}^{\log,+}-F_{2,j,k_1,k_2}^+=0,\\
&L_-\tilde{P}_{1,j,k_1,k_2}^-+\left(2j+(k_1+k_2+1)\alpha\right)\tilde{P}_{1,j,k_1,k_2}^++\frac{1}{|y|^{2\sigma}}F_{1,j,k_1,k_2}^{\sigma,-}+\frac{1}{|y|^{2\sigma}}\log|y|F_{1,j,k_1,k_2}^{\log,-}-F_{1,j,k_1,k_2}^-=0,\\
&L_-\tilde{P}_{2,j,k_1,k_2}^-+\left(2j+(k_1+k_2+1)\alpha\right)\tilde{P}_{2,j,k_1,k_2}^++\frac{1}{|y|^{2\sigma}}F_{1,j,k_1,k_2}^{\sigma,-}+\frac{1}{|y|^{2\sigma}}\log|y|F_{1,j,k_1,k_2}^{\log,-}-F_{2,j,k_1,k_2}^+=0
\end{empheq}
and
\begin{empheq}[left={(S'_{j,k_1,k_2})\ \empheqlbrace\ }]{align*}
&P_{\bullet,j,k_1,k_2}^+=\tilde{P}_{\bullet,j,k_1,k_2}^+-\frac{c_{\bullet,j,k_1,k_2}^+}{2}\Lambda Q,\\
&P_{\bullet,j,k_1,k_2}^-=\tilde{P}_{\bullet,j,k_1,k_2}^--c_{\bullet,j,k_1,k_2}^-Q-\frac{(2j+(k_1+k_2+1)\alpha)c_{\bullet,j,k_1,k_2}^+}{8}|y|^2Q.
\end{empheq}
Then by applying $(S'_{j,k_1,k_2})$ to a solution for $(\tilde{S}_{j,k_1,k_2})$, we obtain a solution for $(S_{j,k_1,k_2})$.

Firstly, we solve
\begin{empheq}[left={(\tilde{S}_{0,0,0})\ \empheqlbrace\ }]{align*}
&L_+\tilde{P}_{1,0,0,0}^++\beta_{1,0,0,0}\frac{|y|^2}{4}Q+\frac{1}{|y|^{2\sigma}}Q=0,\\
&L_+\tilde{P}_{2,0,0,0}^+-\beta_{2,0,0,0}\frac{|y|^2}{4}Q+\frac{1}{|y|^{2\sigma}}\log|y|Q=0,\\
&L_-\tilde{P}_{\bullet,0,0,0}^-+\alpha \tilde{P}_{\bullet,0,0,0}^+=0.
\end{empheq}
For any $\beta_{1,0,0,0}\in\mathbb{R}$, there exists a solution $\tilde{P}_{1,0,0,0}^+\in\mathcal{Y}$. Let
\[
\beta_{1,0,0,0}:=\frac{4\sigma\||\cdot|^{-\sigma}Q\|_2^2}{\||\cdot|Q\|_2^2}.
\]
Then since
\begin{align*}
\left(\tilde{P}_{1,0,0,0}^+,Q\right)_2=-\frac{1}{2}\left\langle L_+\tilde{P}_{1,0,0,0}^+,\Lambda Q\right\rangle=-\frac{1}{2}\left(\frac{\beta_{0,0}}{4}\||\cdot|Q\|_2^2-\sigma\||\cdot|^{-\sigma}Q\|_2^2\right)=0,
\end{align*}
there exists a solution $\tilde{P}_{1,0,0,0}^-\in\mathcal{Y}$. By taking $c_{1,0,0,0}^+=0$, we obtain a solution $(P_{1,0,0,0}^+,P_{1,0,0,0}^-,\beta_{1,0,0,0},c_{1,0,0,0}^+)\in\mathcal{Y}^2\times\mathbb{R}^2$. Similarly, we obtain a solution $(P_{2,0,0,0}^+,P_{2,0,0,0}^-,\beta_{2,0,0,0},c_{2,0,0,0}^+)\in\mathcal{Y}^2\times\mathbb{R}^2$. Here, let $H(j_0,k_{1,0},k_{2,0})$ denote by that
\begin{align*}
&\forall (j,k_1,k_2)\in\Sigma_{K+K'},\ k_2<k_{2,0}\ \mbox{or}\ (k_2=k_{2,0}\ \mbox{and}\ k_1<k_{1,0})\ \mbox{or}\ (k_2=k_{2,0}\ \mbox{and}\ k_1=k_{1,0}\ \mbox{and}\ j<j_0)\\
&\hspace{100pt}\Rightarrow (S_{j,k_1,k_2})\ \mbox{has a solution}\ (P_{\bullet,j,k_1,k_2}^+,P_{\bullet,j,k_1,k_2}^-,\beta_{\bullet,j,k_1,k_2},c_{\bullet,j,k_1,k_2}^+)\in\mathcal{Y}^2\times\mathbb{R}^2.
\end{align*}
From the above discuss, $H(1,0,0)$ is true. If $H(j_0,k_{1,0},k_{2,0})$ is true, then $F_{\bullet,j_0,k_{1,0},k_{2,0}}^\pm$ is defined and belongs to $\mathcal{Y}$. Moreover, for any $\beta_{\bullet,j_0,k_{1,0},k_{2,0}}$, there exists a solution $\tilde{P}_{\bullet,j_0,k_{1,0},k_{2,0}}^+$. Let be $\beta_{\bullet,j_0,k_{1,0},k_{2,0}}$ such that
\[
\left\langle\left(2j+(k_1+k_2+1)\alpha\right)\tilde{P}_{\bullet,j,k_1,k_2}^++\frac{1}{|y|^{2\sigma}}F_{\bullet,j,k_1,k_2}^{\sigma,-}+\frac{1}{|y|^{2\sigma}}\log|y|F_{\bullet,j,k_1,k_2}^{\log,-}-F_{\bullet,j,k_1,k_2}^-,Q\right\rangle=0.
\]
Then we obtain a solution $\tilde{P}_{\bullet,j_0,k_{1,0},k_{2,0}}^-$. Here, we define
\begin{align*}
c_{\bullet,j_0,k_{1,0},k_{2,0}}^-&:=\left\{
\begin{array}{cc}
\frac{\tilde{P}_{\bullet,j_0,k_{1,0},k_{2,0}}^-(0)}{Q(0)}&(j_0+k_{1,0}+k_{2,0}\neq K+1),\\
0&(j_0+k_{1,0}+k_{2,0}=K+1,\mbox{\ and\ }\tilde{P}_{j_0,k_{1,0},k_{2,0}}^-(0)\neq 0),\\
1&(j_0+k_{1,0}+k_{2,0}=K+1,\mbox{\ and\ } \tilde{P}_{j_0,k_{1,0},k_{2,0}}^-(0)=0),
\end{array}\right.\\
c_{\bullet,j_0,k_{1,0},k_{2,0}}^+&:=\left\{
\begin{array}{cc}
0&(j_0+k_{1,0}+k_{2,0}\leq K),\\
0&(j_0+k_{1,0}+k_{2,0}=K+1,\mbox{\ and\ }\tilde{P}_{j_0,k_{1,0},k_{2,0}}^+(0)\neq 0),\\
1&(j_0+k_{1,0}+k_{2,0}=K+1,\mbox{\ and\ }\tilde{P}_{j_0,k_{1,0},k_{2,0}}^+(0)=0),\\
\frac{2\tilde{P}_{\bullet,j_0,k_{1,0},k_{2,0}}^+(0)}{Q(0)}&(j_0+k_{1,0}+k_{2,0}\geq K+2).
\end{array}\right.
\end{align*}
Then we obtain a solution for $(S_{j_0,k_{1,0},k_{2,0}})$. This means that $H(j_0+1,k_{1,0},k_{2,0})$ is true if $j_0+k_{1,0}+k_{2,0}\leq K+K'-1$, $H(0,k_{1,0}+1,k_{2,0})$ is true if $j_0+k_{1,0}+k_{2,0}=K+K'$ and $k_{1,0}+k_{2,0}\leq K+K'-1$, and $H(0,0,k_{2,0}+1)$ is true if $j_0+k_{1,0}+k_{2,0}=k_{1,0}+k_{2,0}=K+K'$. In particular, $H(0,0,K+K'+1)$ means that for any $(j,k_1,k_2)\in\Sigma_{K+K'}$, there exists a solution $(P_{\bullet,j,k_1,k_2}^+,P_{\bullet,j,k_1,k_2}^-,\beta_{\bullet,j,k_1,k_2},c_{\bullet,j,k_1,k_2}^+)\in\mathcal{Y}^2\times\mathbb{R}^2$.

Furthermore,  $P_{\bullet,j,k_1,k_2}^\pm(0)\neq 0$ for $j+k_1+k_2=K+1$ and $P_{\bullet,j,k_1,k_2}^\pm(0)=0$ for $j+k_1+k_2\geq K+2$ hold.
\end{proof}

\begin{proposition}
\label{pconti}
For $P_{\bullet,j,k_1,k_2}^\pm$,
\[
\Lambda P_{\bullet,j,k_1,k_2}^\pm\in H^1(\mathbb{R}^N)\cap C(\mathbb{R}^N).
\]
Namely, $P_{\bullet,j,k_1,k_2}^\pm\in\mathcal{Y}'$.
\end{proposition}

\begin{proof}
For the proof, see \cite{NI}.
\end{proof}

\begin{proposition}
\label{Pint}
For $P_{\bullet,0,k_1,k_2}^{\pm}$ with $k_1+k_2=K+K'$,
\[
\frac{1}{r^2}P_{\bullet,0,k_1,k_2}^\pm,\frac{1}{r}\frac{\partial P_{\bullet,0,k_1,k_2}^\pm}{\partial r}\in L^\infty(\mathbb{R}^N),
\]
where $r=|y|$.
\end{proposition}

\begin{proof}
We prove only for $P_{1,0,k_1,k_2}^+$. See \cite{LMR,NI} for details of proofs.

Let $f_{k_1,k_2}:=P_{1,0,k_1,k_2}^+$ for $k_1,k_2\in\mathbb{N}_0$. Then $f_{k_1,k_2}(0)\neq0$ if $k_1+k_2=K+1$ and $f_{k_1,k_2}(0)=0$ for $k_1+k_2\geq K+2$ hold. Moreover, Let
\[
F_{k_1,k_2}:=f_{k_1,k_2}-\left(\frac{4}{N}+1\right)Q^\frac{4}{N}f_{k_1,k_2}+\beta_{1,0,k_1,k_2}\frac{|y|^2}{4}Q-F_{1,0,k_1,k_2}^+-c_{1,0,k_1,k_2}^+Q.
\]
Then
\begin{align}
\label{fkeq}
\frac{1}{r^{N-1}}\frac{\partial}{\partial r}\left(r^{N-1}\frac{\partial f_{k_1,k_2}}{\partial r}\right)=\frac{1}{r^{2\sigma}}f_{k_1-1,k_2}+\frac{\log r}{r^{2\sigma}}f_{k_1,k_2-1}+F_{k_1,k_2}
\end{align}
holds. If $k_2=0$, then we obtain conclusion for $P_{1,0,K+K',0}^+$ as in \cite{NI}.

If $r^{-q}(\log r)^{q'}f_{k_1-1,k_2}$ and $r^{-q}(\log r)^{q'}f_{k_1,k_2-1}$ converge to non-zero as $r\rightarrow+0$ for some $q\in[0,2\sigma]$ and $q'\geq 0$ or $r^{-q}(\log r)^{q'}f_{k_1-1,k_2}$ and $q'\geq 0$ or $r^{-q}(\log r)^{q'}f_{k_1,k_2-1}$ converge as $r\rightarrow+0$ for some $q>2\sigma$ and $q'\geq0$, then $r^{N-1}\frac{\partial f_{k+1}}{\partial r}$ converges to $0$ as $r\rightarrow+0$.

Let $\sigma_0=\sigma'_0:=0$ and $C_{k_1,k_2}:=f_{k_1,k_2}(0)$ for $k_1+k_2=K+1$. Moreover, let
\begin{align*}
k&:=k_1+k_2-K-1,&\sigma_{k+1}&:=\left\{
\begin{array}{cc}
1-\sigma+\sigma_k&\left(\sigma_k<\sigma\right)\\
1&\left(\sigma_k\geq\sigma\right)
\end{array}\right.,\\
\sigma'_k&:=\left\{
\begin{array}{cc}
\sigma'_{k-1}+1&\left(\sigma_{k-1}<\sigma\right)\\
0&\left(\sigma_{k-1}\geq \sigma\right)
\end{array}\right.,&C_{k_1,k_2}&:=\left\{
\begin{array}{cc}
\frac{C_{k_1,k_2-1}}{2\sigma_k(N-2(\sigma-\sigma_{k-1}))}&\left(\sigma_{k-1}<\sigma\right)\\
\frac{0^{2(\sigma'_{k-1}-\sigma)}C_{k_1,k_2-1}+F_{k_1,k_2}(0)}{2N}&\left(\sigma_{k-1}=\sigma\right)\\
\frac{F_{k_1,k_2}(0)}{2N}&\left(\sigma_{k-1}>\sigma\right)
\end{array}\right..
\end{align*}
In particular, if $\sigma_{k-1}<\sigma$, then $C_{k_1,k_2}\neq 0$. Then
\begin{eqnarray}
\label{origindecay}
\lim_{r\rightarrow+0}\frac{1}{r^{2\sigma_k}(\log r)^{\sigma'_k}}f_{k_1,k_2}(r)=C_{k_1,k_2}
\end{eqnarray}
holds. For $k=0$, it clearly holds. Moreover, for $k\geq 1$,
\[
\lim_{r\rightarrow+0}\frac{1}{r^{2\sigma_k-1}(\log r)^{\sigma'_k}}\frac{\partial f_{k_1,k_2}}{\partial r}(r)=2\sigma_kC_{k_1,k_2}
\]
holds. Consequently, we obtain Proposition \ref{Pint} if $K'$ is sufficiently large.
\end{proof}

\section{Decomposition of functions}
\label{sec:decomposition}
The parameters $\tilde{\lambda},\tilde{b},\tilde{\gamma}$ to be used for modulation are obtained by the following lemma:

\begin{lemma}[Decomposition]
\label{decomposition}
For any sufficiently small constants $\overline{l},\overline{\epsilon}>0$, there exist a sufficiently small $\delta>0$ such that the following statement holds. Let $I$ be an interval. We assume that $u\in C(I,H^1(\mathbb{R}^N))\cap C^1(I,H^{-1}(\mathbb{R}^N))$ satisfies
\[
\forall\ t\in I,\ \left\|\lambda(t)^{\frac{N}{2}}u\left(t,\lambda(t)y\right)e^{i\gamma(t)}-Q\right\|_{H^1}< \delta
\]
for some functions $\lambda:I\rightarrow(0,\overline{l})$ and $\gamma:I\rightarrow\mathbb{R}$. Then there exist unique functions $\tilde{\lambda}:I\rightarrow(0,\infty)$, $\tilde{b}:I\rightarrow\mathbb{R}$, and $\tilde{\gamma}:I\rightarrow\mathbb{R}\slash 2\pi\mathbb{Z}$ such that 
\begin{align}
\label{mod}
&u(t,x)=\frac{1}{\tilde{\lambda}(t)^{\frac{N}{2}}}\left(P+\tilde{\varepsilon}\right)\left(t,\frac{x}{\tilde{\lambda}(t)}\right)e^{-i\frac{\tilde{b}(t)}{4}\frac{|x|^2}{\tilde{\lambda}(t)^2}+i\tilde{\gamma}(t)},\\
&\left|\frac{\tilde{\lambda}(t)}{\lambda(t)}-1\right|+\left|\tilde{b}(t)\right|+\left|\tilde{\gamma}(t)-\gamma(t)\right|_{\mathbb{R}\slash 2\pi\mathbb{Z}}<\overline{\epsilon}\nonumber
\end{align}
hold, where $|\cdot|_{\mathbb{R}\slash 2\pi\mathbb{Z}}$ is defined by
\[
|c|_{\mathbb{R}\slash 2\pi\mathbb{Z}}:=\inf_{m\in\mathbb{Z}}|c+2\pi m|,
\]
and that $\tilde{\varepsilon}$ satisfies the orthogonal conditions
\begin{align}
\label{orthocondi}
\left(\tilde{\varepsilon},i\Lambda P\right)_2=\left(\tilde{\varepsilon},|y|^2P\right)_2=\left(\tilde{\varepsilon},i\rho\right)_2=0
\end{align}
on $I$. In particular, $\tilde{\lambda}$, $\tilde{b}$, and $\tilde{\gamma}$ are $C^1$ functions and independent of $\lambda$ and $\gamma$.
\end{lemma}

\section{Approximate blow-up law}
\label{sec:apppara}
In this section, we describe the initial values and the approximation functions of the parameters $\lambda$ and $b$ in the decomposition.

We expect the parameters $\lambda$ and $b$ in the decomposition to approximately satisfy
\[
\frac{1}{\lambda}\frac{\partial \lambda}{\partial s}+b=\frac{\partial b}{\partial s}+b^2-\theta=0.
\]
Therefore, the approximation functions $\lambda_{\mathrm{app}}$ and $b_{\mathrm{app}}$ of the parameters $\lambda$ and $b$ will be determined by the following lemma:

Let $\lambda_0>0$ be sufficiently small and define $J$ by
\[
J(\lambda):=\int_\lambda^{\lambda_0}\frac{1}{\mu^{\frac{\alpha}{2}+1}\sqrt{-\frac{2\beta_1}{(2-\alpha)^2}+\frac{2\beta_2}{2-\alpha}-\frac{2\beta_1}{2-\alpha}\log\mu}}d\mu
\]
for any $\lambda\in(0,\lambda_0)$, where $\beta_2:=\beta_{2,0,0,0}$. In particular, since
\begin{align*}
&\sigma=\frac{2-\alpha}{2},\quad \beta_1=\frac{4\sigma\||\cdot|^{-\sigma}Q\|_2^2}{\||\cdot|Q\|_2^2},\quad \beta_2=\frac{4}{\||\cdot|Q\|_2^2}\left(\sigma\int_{\mathbb{R}^N}\frac{1}{|y|^{2\sigma}}\log|y|Q^2dy-\frac{1}{2}\||\cdot|^{-\sigma}Q\|_2^2\right),\\
&\beta'_1=\frac{4}{\||\cdot|Q\|_2^2}\int_{\mathbb{R}^N}\frac{1}{|y|^{2\sigma}}\log|y|Q^2dy.
\end{align*}
we obtain
\[
\beta'_1=-\frac{2\beta_1}{(2-\alpha)^2}+\frac{2\beta_2}{2-\alpha}.
\]

\begin{lemma}
\label{lbsol}
Let
\[
\lambda_{\app}(s):=J^{-1}(s),\quad b_{\app}(s):=\lambda_{\app}(s)^{\frac{\alpha}{2}}\sqrt{\beta'_1-\frac{2\beta_1}{2-\alpha}\log\lambda_{\app}(s)}.
\]
Then $(\lambda_{\app},b_{\app})$ is a solution for
\[
\frac{\partial b}{\partial s}+b^2+\beta_1\lambda^\alpha\log\lambda-\beta_2\lambda^\alpha=0,\quad \frac{1}{\lambda}\frac{\partial \lambda}{\partial s}+b=0
\]
in $s>0$.
\end{lemma}

\begin{proof}
This result is proven via direct calculation. For the construction of $(\lambda_{\app},b_{\app})$, see \cite{LMR}.
\end{proof}

For the sake of the discussion that follows, we will introduce some properties of the Lambert function $W_{-1}$. For definition of $W_{-1}$, see \cite{LW}.

\begin{corollary}[\cite{IC}]
\label{Wprop}
For any $u>0$,
\[
(1-\epsilon)u-\frac{2}{\epsilon}<-W_{-1}(-e^{-u-1})-1-\sqrt{2u}<u
\]
holds. In particular, for any $z\in\left(0,\frac{1}{e}\right)$,
\[
-(1-\epsilon)\log z+\sqrt{2(-\log z-1)}+\epsilon-\frac{2}{\epsilon}<-W_{-1}(-z)<-\log z+\sqrt{2(-\log z-1)}
\]
holds.
\end{corollary}

\begin{lemma}
\label{Jprop}
Let $\lambda$ be sufficiently small. Then
\begin{align}
\label{Japp}
J(\lambda)^{-1}=\frac{2}{\alpha}\sqrt{\frac{2\beta_1}{2-\alpha}}\lambda^\frac{\alpha}{2}\sqrt{|\log\lambda|}\left(1+O\left(\frac{1}{|\log\lambda|}\right)\right)
\end{align}
holds. In addition,
\[
\left|\frac{\alpha^2}{4}\left(-\frac{2\beta_1}{2-\alpha}\lambda^\alpha\log\lambda+\beta'_1\lambda^\alpha\right)-J(\lambda)^{-2}\right|\lesssim \lambda^\alpha
\]
holds.

Moreover, let $s$ be sufficiently large. Then
\begin{align}
&s^{-2}=J(\lambda_{\app}(s))^{-2}=\frac{4}{\alpha^2}\frac{2\beta_1}{2-\alpha}\lambda_{\app}(s)^\alpha|\log\lambda_{\app}(s)|(1+o(1)),\nonumber\\
\label{lambapp}
&\lambda_{\app}(s)^\alpha|\log\lambda_{\app}(s)|\approx b_{\app}(s)^2\approx s^{-2},\quad \lambda_{\app}(s)\approx s^{-\frac{2}{\alpha}}\left(\log s\right)^{-\frac{1}{\alpha}}
\end{align}
hold.
\end{lemma}

\begin{proof}
Since $\lambda_0$ is sufficiently small,
\[
\mu^{\frac{\alpha}{2}+1}\sqrt{\frac{\beta_1}{2-\alpha}|\log\mu|}\leq \mu^{\frac{\alpha}{2}+1}\sqrt{-\frac{2\beta_1}{(2-\alpha)^2}+\frac{2\beta_2}{2-\alpha}-\frac{2\beta_1}{2-\alpha}\log\mu}\leq \mu^{\frac{\alpha}{2}+1}\sqrt{\frac{4\beta_1}{2-\alpha}|\log\mu|}
\]
holds for any $\mu\in(0,\lambda_0)$. Therefore, we obtain
\[
\frac{\alpha}{2}\frac{1}{\sqrt{\frac{4\beta_1}{2-\alpha}|\log\lambda|}}\left(\frac{1}{\lambda^\frac{\alpha}{2}}-\frac{1}{\lambda_0^\frac{\alpha}{2}}\right)\leq\int_{\lambda}^{\lambda_0}\frac{1}{\mu^{\frac{\alpha}{2}+1}\sqrt{-\frac{2\beta_1}{(2-\alpha)^2}+\frac{2\beta_2}{2-\alpha}-\frac{2\beta_1}{2-\alpha}\log\mu}}d\mu\leq\frac{\alpha}{4}\frac{1}{\sqrt{\frac{\beta_1}{2-\alpha}\lambda^\alpha|\log\lambda|}}.
\]
Moreover, we obtain
\[
\frac{\alpha}{4}\frac{1}{\sqrt{\frac{4\beta_1}{2-\alpha}|\log\lambda|}}\frac{1}{\lambda^\frac{\alpha}{2}}\leq J(\lambda)\leq \frac{\alpha}{4}\frac{1}{\sqrt{\frac{\beta_1}{2-\alpha}\lambda^\alpha|\log\lambda|}}
\]
for any sufficiently small $\lambda$.

Next, since
\[
\frac{1}{\frac{2}{\alpha}\sqrt{\frac{2\beta_1}{2-\alpha}}\lambda^\frac{\alpha}{2}\sqrt{|\log\lambda|}}=\int_\lambda^{\lambda_0}\left(\frac{1}{\sqrt{\frac{2\beta_1}{2-\alpha}}\mu^{\frac{\alpha}{2}+1}\sqrt{|\log\mu|}}-\frac{1}{\frac{4}{\alpha}\sqrt{\frac{2\beta_1}{2-\alpha}}\mu^{\frac{\alpha}{2}+1}|\log\mu|^\frac{3}{2}}\right)d\mu+\frac{1}{\frac{2}{\alpha}\sqrt{\frac{2\beta_1}{2-\alpha}}\lambda_0^\frac{\alpha}{2}\sqrt{|\log\lambda_0|}},
\]
we obtain
\begin{align*}
&J(\lambda)-\frac{1}{\frac{2}{\alpha}\sqrt{\frac{2\beta_1}{2-\alpha}}\lambda^\frac{\alpha}{2}\sqrt{|\log\lambda|}}\\
=&\int_\lambda^{\lambda_0}\frac{-\beta'_1}{\mu^{\frac{\alpha}{2}+1}\sqrt{\frac{2\beta_1}{2-\alpha}|\log\mu|+\beta'_1}\sqrt{\frac{2\beta_1}{2-\alpha}|\log\mu|}\left(\sqrt{\frac{2\beta_1}{2-\alpha}|\log\mu|+\beta'_1}+\sqrt{\frac{2\beta_1}{2-\alpha}|\log\mu|}\right)}d\mu\\
&\hspace{20pt}+\int_\lambda^{\lambda_0}\frac{1}{\frac{4}{\alpha}\sqrt{\frac{2\beta_1}{2-\alpha}}\mu^{\frac{\alpha}{2}+1}|\log\mu|^\frac{3}{2}}d\mu+\frac{1}{\frac{2}{\alpha}\sqrt{\frac{2\beta_1}{2-\alpha}}\lambda_0^\frac{\alpha}{2}\sqrt{|\log\lambda_0|}}.
\end{align*}
Therefore,
\[
\left|J(\lambda)-\frac{1}{\frac{2}{\alpha}\sqrt{\frac{2\beta_1}{2-\alpha}}\lambda^\frac{\alpha}{2}\sqrt{|\log\lambda|}}\right|\lesssim\int_\lambda^{\lambda_0}\frac{1}{\mu^{\frac{\alpha}{2}+1}|\log\mu|^\frac{3}{2}}d\mu+1\lesssim\frac{1}{\lambda^{\frac{\alpha}{2}}|\log\lambda|^\frac{3}{2}}.
\]
Accordingly,
\[
\left|J(\lambda)^{-1}-\frac{2}{\alpha}\sqrt{\frac{2\beta_1}{2-\alpha}}\lambda^\frac{\alpha}{2}\sqrt{|\log\lambda|}\right|\lesssim \lambda^\alpha|\log\lambda|\frac{1}{\lambda^{\frac{\alpha}{2}}|\log\lambda|^\frac{3}{2}}=\lambda^\frac{\alpha}{2}\sqrt{|\log\lambda|}\frac{1}{|\log\lambda|}.
\]

Since $\lambda_{\app}(s)\rightarrow 0$ as $s\rightarrow\infty$,
\begin{align*}
s^{-1}=J(\lambda_{\app}(s))^{-1}=\frac{2}{\alpha}\sqrt{\frac{2\beta_1}{2-\alpha}}\lambda_{\app}(s)^\frac{\alpha}{2}\sqrt{|\log\lambda_{\app}(s)|}\left(1+o(1)\right)
\end{align*}
for any sufficiently large $s$. Therefore,
\[
\frac{1}{2}s^{-2}\leq\frac{4}{\alpha^2}\frac{2\beta_1}{2-\alpha}\lambda_{\app}(s)^\alpha|\log\lambda_{\app}(s)|\leq 2s^{-2}
\]
and
\[
\frac{\beta_1}{2-\alpha}\lambda_{\app}(s)^\alpha|\log\lambda_{\app}(s)|\leq b_{\app}(s)^2\leq \frac{4\beta_1}{2-\alpha}\lambda_{\app}(s)^\alpha|\log\lambda_{\app}(s)|
\]
hold.

Finally, since
\[
-C_1s^{-2}\leq\lambda_{\app}(s)^\alpha\log\lambda_{\app}(s)^\alpha=\log\lambda_{\app}(s)^\alpha e^{\log\lambda_{\app}(s)^\alpha}=W_{-1}^{-1}(\log\lambda_{\app}(s)^\alpha)\leq -C_2s^{-2},
\]
we obtain
\[
W_{-1}(-C_2s^{-2})\leq \log\lambda_{\app}(s)^\alpha\leq W_{-1}(-C_1s^{-2}).
\]
Since $e^{W(z)}=\frac{z}{W(z)}$, we obtain
\[
\frac{-C_2s^{-2}}{W_{-1}(-C_2s^{-2})}\leq \lambda_{\app}(s)^\alpha\leq \frac{-C_2s^{-2}}{W_{-1}(-C_2s^{-2})}.
\]
Since
\[
-\frac{1}{2}\left(-2\log s+\log C\right)<-W_{-1}\left(-Cs^{-1}\right)<-2\left(-2\log s+\log C\right)
\]
from Lemma \ref{Wprop}, we obtain
\[
\lambda_{\app}(s)\approx s^{-\frac{2}{\alpha}}(\log s)^{-\frac{1}{\alpha}}.
\]
\end{proof}

Furthermore, the following lemma determines $\lambda(s_1)$ and $b(s_1)$ for a given energy level $E_0$ and a sufficiently large $s_1$.

\begin{lemma}
\label{paraini}
Let define $C_0:=\frac{8E_0}{\||y|Q\|_2^2}$ and $0<\lambda_0\ll 1$ such that $-\frac{2\beta_1}{2-\alpha}\log\lambda_0+\beta'_1+C_0{\lambda_0}^{2-\alpha}>0$. For $\lambda\in(0,\lambda_0]$, we set
\[
\mathcal{F}(\lambda):=\int_\lambda^{\lambda_0}\frac{1}{\mu^{\frac{\alpha}{2}+1}\sqrt{-\frac{2\beta_1}{2-\alpha}\log\mu+\beta'_1+C_0\mu^{2-\alpha}}}d\mu.
\]
Then for any $s_1\gg 1$, there exist $b_1,\lambda_1>0$ such that
\[
\left|\frac{{s_1}^{-1}}{J(\lambda_1)^{-1}}-1\right|+\left|\frac{b_1}{b_{\mathrm{app}}(s_1)}-1\right|\lesssim {s_1}^{-\frac{1}{2}}\sqrt{|\log s_1|}+{s_1}^{2-\frac{4}{\alpha}}\sqrt{|\log s_1|},\quad \mathcal{F}(\lambda_1)=s_1,\quad E(P_{\lambda_1,b_1,\gamma})=E_0.
\]
Moreover,
\[
\left|\mathcal{F}(\lambda)-J(\lambda)\right|\lesssim\lambda^{-\frac{\alpha}{4}}+\lambda^{2-\frac{3}{2}\alpha}
\]
holds.
\end{lemma}

\begin{proof}
The method of choosing $\lambda_1$ and $b_1$ and the estimate of $\mathcal{F}$ and $b_1$ are the same as in \cite{LMR,NI}.

Since
\[
s_1=J(\lambda_{\app}(s_1)),\quad J(\lambda)^{-1}\approx\lambda^\frac{\alpha}{2}\sqrt{|\log\lambda|},
\]
we obtain
\[
\left|\frac{J(\lambda_{\app}(s_1))}{J(\lambda_1)}-1\right|\lesssim \lambda_1^{\frac{\alpha}{4}}\sqrt{|\log\lambda_1|}+\lambda_1^{2-\alpha}\sqrt{|\log\lambda_1|}.
\]
Moreover, since
\[
s_1=\mathcal{F}(\lambda_1)\lesssim\int_\lambda^{\lambda_0}\frac{1}{\mu^{\frac{\alpha}{2}+1}\sqrt{|\log\mu|}}d\mu\leq\int_\lambda^{\lambda_0}\frac{1}{\mu^{\frac{\alpha}{2}+1}}d\mu\leq \lambda_1^{-\frac{\alpha}{2}},
\]
we obtain
\[
\lambda_1\lesssim s_1^{-\frac{2}{\alpha}}.
\]
Consequently, we obtain the conclusion.
\end{proof}

\section{Uniformity estimates for decomposition}
\label{sec:uniesti}
In this section, we estimate \textit{modulation terms}.

We define
\[
t_{\app}(s):=-\int_s^\infty \lambda_{\app}(\mu)^2d\mu.
\]
For $t_1<0$ such that is sufficiently close to $0$, we define
\[
s_1:={t_{\app}}^{-1}(t_1).
\]
Additionally, let $\lambda_1$ and $b_1$ be given in Lemma \ref{paraini} for $s_1$ and $\gamma_1=0$. Let $u$ be the solution for \eqref{NLS} with $\pm=+$ with an initial value
\[
u(t_1,x):=P_{\lambda_1,b_1,0}(x).
\]
Then since $u$ satisfies the assumption of Lemma \ref{decomposition} in a neighbourhood of $t_1$, there exists a decomposition $(\tilde{\lambda}_{t_1},\tilde{b}_{t_1},\tilde{\gamma}_{t_1},\tilde{\varepsilon}_{t_1})$ such that $(\ref{mod})$ in a neighbourhood $I$ of $t_1$. The rescaled time $s_{t_1}$ is defined by
\[
s_{t_1}(t):=s_1-\int_t^{t_1}\frac{1}{\tilde{\lambda}_{t_1}(\tau)^2}d\tau.
\]
Then we define an inverse function ${s_{t_1}}^{-1}:s_{t_1}(I)\rightarrow I$. Moreover, we define
\begin{align*}
t_{t_1}&:={s_{t_1}}^{-1},& \lambda_{t_1}(s)&:=\tilde{\lambda}(t_{t_1}(s)),& b_{t_1}(s)&:=\tilde{b}(t_{t_1}(s)),\\
\gamma_{t_1}(s)&:=\tilde{\gamma}(t_{t_1}(s)),& \varepsilon_{t_1}(s,y)&:=\tilde{\varepsilon}(t_{t_1}(s),y).&&
\end{align*}
For the sake of clarity in notation, we often omit the subscript $t_1$. In particular, it should be noted that $u\in C((T_*,T^*),\Sigma^2(\mathbb{R}^N))$ and $|x|\nabla u\in C((T_*,T^*),L^2(\mathbb{R}^N))$. Furthermore, let $I_{t_1}$ be the maximal interval such that a decomposition as $(\ref{mod})$ is obtained and we define 
\[
J_{s_1}:=s\left(I_{t_1}\right).
\]
Additionally, let $s_0\ (\leq s_1)$ be sufficiently large and let 
\[
s':=\max\left\{s_0,\inf J_{s_1}\right\}.
\]

Let $s_*$ be defined by
\[
s_*:=\inf\left\{\sigma\in(s',s_1]\ \middle|\ \mbox{(\ref{bootstrap}) holds on }[\sigma,s_1]\right\},
\]
where
\begin{align}
\label{bootstrap}
&\left\|\varepsilon(s)\right\|_{H^1}^2+b(s)^2\||y|\varepsilon(s)\|_2^2<s^{-2K},\quad \left|\frac{s^{-1}}{J(\lambda(s))^{-1}}-1\right|+\left|\frac{b(s)}{b_{\mathrm{app}}(s)}-1\right|<(\log s)^{-\frac{1}{2}}.
\end{align}

Finally, we define
\[
\Mod(s):=\left(\frac{1}{\lambda}\frac{\partial \lambda}{\partial s}+b,\frac{\partial b}{\partial s}+b^2-\theta,1-\frac{\partial \gamma}{\partial s}\right).
\]
The goal of this section is to estimate of $\mathrm{Mod}(s)$.

\begin{lemma}
\label{lambapprox}
On $(s_*,s_1]$,
\[
\lambda(s)^\alpha|\log\lambda(s)|\approx s^{-2},\quad b(s)\approx s^{-1}
\]
holds.
\end{lemma}

\begin{proof}
For $b$, it is clear from \eqref{lambapp} and \eqref{bootstrap}.

From Lemma \ref{decomposition}, we may assume $|\lambda(s)|\leq \frac{1}{2}$. Since
\[
\left|\frac{J(\lambda(s))^{-1}}{s^{-1}}-1\right|\lesssim (\log s)^{-\frac{1}{2}},
\]
we obtain
\[
\left|J(\lambda(s))^{-1}\right|\lesssim s^{-1}\left(1+(\log s)^{-\frac{1}{2}}\right).
\]
Moreover, since
\[
\lambda(s)^\alpha\lesssim\lambda(s)^\alpha|\log\lambda(s)|\lesssim s^{-2}\left(1+(\log s)^{-\frac{1}{2}}\right)^2
\]
from \eqref{Japp}, we have $\lambda(s)=o(1)$ and
\[
\frac{4}{\alpha^2}\frac{2\beta_1}{2-\alpha}\lambda(s)^\alpha|\log\lambda(s)|=J(\lambda(s))^{-2}(1+o(1)).
\]

Next, since
\[
\left|\frac{J(\lambda(s))^{-2}}{s^{-2}}-1\right|\lesssim (\log s)^{-\frac{1}{2}},\quad\mbox{i.e.,}\ \left|J(\lambda(s))^{-2}-s^{-2}\right|\lesssim s^{-2}(\log s)^{-\frac{1}{2}},
\]
we obtain
\[
\left|\frac{4}{\alpha^2}\frac{2\beta_1}{2-\alpha}\lambda(s)^\alpha|\log\lambda(s)|-s^{-2}\right|\leq J(\lambda(s))^{-2}o(1)+s^{-2}(\log s)^{-\frac{1}{2}}=o\left(s^{-2}\right).
\]
Therefore, we obtain
\[
\frac{4}{\alpha^2}\frac{2\beta_1}{2-\alpha}\lambda(s)^\alpha|\log\lambda(s)|=s^{-2}\left(1+o(1)\right).
\]
\end{proof}

\section{Modified energy function}
\label{sec:MEF}
We proceed with a modified version of the technique presented in Le Coz, Martel, and Rapha\"{e}l \cite{LMR} and Rapha\"{e}l and Szeftel \cite{RSEU}. Let $m>0$ be sufficiently large and define
\begin{align*}
H(s,\varepsilon)&:=\frac{1}{2}\left\|\varepsilon\right\|_{H^1}^2+b^2\left\||y|\varepsilon\right\|_2^2-\int_{\mathbb{R}^N}\left(F(P+\varepsilon)-F(P)-dF(P)(\varepsilon)\right)dy\\
&\hspace{20pt}+\frac{1}{2}\lambda^\alpha\log\lambda\int_{\mathbb{R}^N}\frac{1}{|y|^{2\sigma}}|\varepsilon|^2dy+\frac{1}{2}\lambda^\alpha\int_{\mathbb{R}^N}\frac{1}{|y|^{2\sigma}}\log|y||\varepsilon|^2dy,\\
S(s,\varepsilon)&:=\frac{1}{\lambda^m}H(s,\varepsilon).
\end{align*}

In this section, we obtain upper and lower estimates of $S$ and a lower estimate of the time derivative of $S$ for use in Section \ref{sec:bootstrap}.

\begin{lemma}[Estimates of $S$]
\label{Sesti}
For $s\in(s_*,s_1]$, 
\[
\|\varepsilon\|_{H^1}^2+b^2\left\||y|\varepsilon\right\|_2^2+O(s^{-2(K+2)})\lesssim H(s,\varepsilon)\lesssim \|\varepsilon\|_{H^1}^2+b^2\left\||y|\varepsilon\right\|_2^2
\]
hold. Moreover, 
\[
\frac{1}{\lambda^m}\left(\|\varepsilon\|_{H^1}^2+b^2\left\||y|\varepsilon\right\|_2^2+O(s^{-2(K+2)})\right)\lesssim S(s,\varepsilon)\lesssim \frac{1}{\lambda^m}\left(\|\varepsilon\|_{H^1}^2+b^2\left\||y|\varepsilon\right\|_2^2\right)
\]
hold.
\end{lemma}

\begin{lemma}[Derivative of $S$ in time]
\label{Sdiff}
For $s\in(s_*,s_1]$, 
\[
\frac{d}{ds}H(s,\varepsilon(s))\gtrsim -b\left(\|\varepsilon\|_{H^1}^2+b^2\left\||y|\varepsilon\right\|_2^2\right)+O(s^{-2(K+2)})
\]
holds. Moreover,
\[
\frac{d}{ds}S(s,\varepsilon(s))\gtrsim \frac{b}{\lambda^m}\left(\|\varepsilon\|_{H^1}^2+b^2\left\||y|\varepsilon\right\|_2^2+O(s^{-(2K+3)})\right)
\]
holds for $s\in(s_*,s_1]$.
\end{lemma}

For the proofs, see \cite{LMR,NI}.

\section{Bootstrap}
\label{sec:bootstrap}
In this section, we use the estimates obtained in Section \ref{sec:MEF} and the bootstrap to establish the estimates of the parameters. Namely, we confirm \eqref{bootstrap} on $[s_0,s_1]$.

\begin{lemma}[Re-estimation]
\label{rebootstrap}
For $s\in(s_*,s_1]$, 
\begin{align}
\label{reepsiesti}
\left\|\varepsilon(s)\right\|_{H^1}^2+b(s)^2\left\||y|\varepsilon(s)\right\|_2^2&\lesssim s^{-(2K+2)},\\
\label{reesti}
\left|\frac{s^{-1}}{J(\lambda(s))^{-1}}-1\right|+\left|\frac{b(s)}{b_{\mathrm{app}}(s)}-1\right|&\lesssim (\log s)^{-1}
\end{align}
holds.
\end{lemma}

\begin{proof}
We can prove $(\ref{reepsiesti})$ as in \cite{LMR}.

Next, since
\[
\left|E(P_{\lambda,b,\gamma}(s))-E_0\right|\leq\left|\int_{s_1}^s\left.\frac{d}{ds}\right|_{s=\tau}E(P_{\lambda,b,\gamma}(s))d\tau\right|\leq\int_s^{s_1}\tau^{-(K+2)+\frac{4}{\alpha}}(\log s)^\frac{2}{\alpha}d\tau\lesssim s^{-(K+1)+\frac{5}{\alpha}},
\]
we have
\begin{align}
&\left|b^2+\frac{2\beta_1}{2-\alpha}\lambda^\alpha\log\lambda-\beta'_1\lambda^\alpha-C_0\lambda^2\right|\nonumber\\
\leq&\lambda^2\left(\left|\frac{b^2}{\lambda^2}+\frac{2\beta_1}{2-\alpha}\lambda^{\alpha-2}\log\lambda-\beta'_1\lambda^{\alpha-2}-\frac{8}{\||y|Q\|_2^2}E(P_{\lambda,b,\gamma})\right|+\frac{8}{\||y|Q\|_2^2}\left|E(P_{\lambda,b,\gamma})-E_0\right|\right)\nonumber\\
\label{bEesti}
\lesssim& s^{-4}.
\end{align}
From the definition of $\mathcal{F}$, we have
\[
\left|\mathcal{F}'(s)-1\right|\lesssim s^{-2},
\]
where $\mathcal{F}(s):=\mathcal{F}(\lambda(s))$. Indeed, since
\begin{align*}
\left|\mathcal{F}'(s)-1\right|&=\frac{1}{\lambda^{\frac{\alpha}{2}}\sqrt{-\frac{2\beta_1}{2-\alpha}\log\lambda+\beta'_1+C_0\lambda^{2-\alpha}}}\left|\frac{1}{\lambda}\frac{\partial\lambda}{\partial s}+b\right|\\
&\hspace{20pt}+\frac{\left|b^2+\frac{2\beta_1}{2-\alpha}\lambda^\alpha\log\lambda-\beta'_1\lambda^\alpha-C_0\lambda^2\right|}{\lambda^{\frac{\alpha}{2}}\sqrt{-\frac{2\beta_1}{2-\alpha}\log\lambda+\beta'_1+C_0\lambda^{2-\alpha}}\left(b+\lambda^{\frac{\alpha}{2}}\sqrt{-\frac{2\beta_1}{2-\alpha}\log\lambda+\beta'_1+C_0\lambda^{2-\alpha}}\right)},
\end{align*}
we obtain
\[
\left|\mathcal{F}'(s)-1\right|\lesssim s^{-(K+1)}+s^{-2}.
\]
Therefore,
\[
\left|s-\mathcal{F}(\lambda(s))\right|\lesssim s^{-1}
\]
holds since $\mathcal{F}(\lambda(s_1))=s_1$. Accordingly, since
\[
\left|s-J(\lambda(s))\right|\leq\left|s-\mathcal{F}(\lambda(s))\right|+\left|J(\lambda(s))-\mathcal{F}(\lambda(s))\right|\lesssim s^\frac{1}{2}(\log s)^\frac{1}{4}+s^{3-\frac{4}{\alpha}}(\log s)^{\frac{3}{2}-\frac{2}{\alpha}},
\]
we obtain
\[
\left|\frac{s^{-1}}{J(\lambda(s))^{-1}}-1\right|\lesssim s^{-\frac{1}{2}}(\log s)^\frac{1}{4}+s^{2-\frac{4}{\alpha}}(\log s)^{\frac{3}{2}-\frac{2}{\alpha}}.
\]

Finally, from \eqref{bEesti}, we have
\begin{align*}
&\left|b(s)^2-{b_{\app}(s)}^2\right|\\
\lesssim&s^{-4}+\lambda(s)^2+\left|\frac{4}{\alpha^2}\left(-\frac{2\beta_1}{2-\alpha}\lambda(s)^\alpha\log\lambda(s)+\beta'_1\lambda(s)^\alpha\right)-J(\lambda(s))^{-2}\right|\\
&\hspace{20pt}+\left|J(\lambda(s))^{-2}-J(\lambda_{\app}(s))^{-2}\right|+\left|\frac{4}{\alpha^2}\left(-\frac{2\beta_1}{2-\alpha}\lambda_{\app}(s)^\alpha\log\lambda_{\app}(s)+\beta'_1\lambda_{\app}(s)^\alpha\right)-J(\lambda_{\app}(s))^{-2}\right|\\
\lesssim&s^{-4}+s^{-\frac{4}{\alpha}}(\log s)^{-\frac{2}{\alpha}}+s^{-2}(\log s)^{-1}+s^{-2}\left(s^{-\frac{1}{2}}(\log s)^\frac{1}{4}+s^{2-\frac{4}{\alpha}}(\log s)^{\frac{3}{2}-\frac{2}{\alpha}}\right)
\end{align*}
and
\[
\left|\frac{b(s)}{b_{\mathrm{app}}(s)}-1\right|\lesssim (\log s)^{-1}+s^{-\frac{1}{2}}(\log s)^\frac{1}{4}+s^{2-\frac{4}{\alpha}}(\log s)^{\frac{3}{2}-\frac{2}{\alpha}}.
\]

Consequently, we obtain \eqref{reesti}.
\end{proof}

\begin{lemma}
\label{s0s*s'}
If $s_0$ is sufficiently large, then $s_*=s'=s_0$.
\end{lemma}

\begin{proof}
This result is proven from Lemma \ref{rebootstrap} and the definitions of $s_*$ and $s'$. See \cite{N} for details of the proof.
\end{proof}

\section{Conversion of estimates}
\label{sec:convesti}
In this section, we rewrite the uniform estimates obtained for the time variable $s$ in Lemma \ref{rebootstrap} into uniform estimates for the time variable $t$.

\begin{lemma}[Interval]
\label{interval}
If $s_0$ is sufficiently large, then there is $t_0<0$ that is sufficiently close to $0$ such that for $t_1\in(t_0,0)$, 
\begin{align*}
[t_0,t_1]\subset {s_{t_1}}^{-1}([s_0,s_1]),\quad \left|t_{\app}(s_{t_1}(t))-t\right|&\lesssim s_{t_1}(t)^{-\frac{4-\alpha}{\alpha}}(\log s_{t_1}(t))^{-\frac{2}{\alpha}-1},\\
\left|t_{\app}(s_{t_1}(t))\right|&\approx s_{t_1}(t)^{-\frac{4-\alpha}{\alpha}}(\log s_{t_1}(t))^{-\frac{2}{\alpha}}\quad (t\in [t_0,t_1])
\end{align*}
holds.
\end{lemma}

\begin{proof}
Since $t_{t_1}(s_1)=t_1=t_{\app}(s_1)$, we have
\begin{align*}
t_{\app}(s)-t_{t_1}(s)&=t_{t_1}(s_1)-t_{t_1}(s)-\left(t_{\app}(s_1)-t_{\app}(s)\right)\\
&=\int_s^{s_1}\lambda_{\mathrm{app}}(\tau)\left(\lambda_{t_1}(\tau)-\lambda_{\mathrm{app}}(\tau)\right)\left(\frac{\lambda_{t_1}(\tau)}{\lambda_{\mathrm{app}}(\tau)}+1\right)d\tau.
\end{align*}
Since $J^{-1}$ is $C^1$ function on $J((0,\lambda_0))$,
\[
\left|\lambda_{t_1}(\tau)-\lambda_{\mathrm{app}}(\tau)\right|=\left|\frac{\tau}{J'(J^{-1}\left(J(\lambda_{t_1}(\tau))+\xi\left(J(\lambda_{\app}(\tau))-J(\lambda_{t_1}(\tau))\right)\right))}\right|\left|\frac{\tau^{-1}}{J(\lambda_{t_1}(\tau))^{-1}}-1\right|
\]
for some $\xi\in[0,1]$. Since
\[
\left|\frac{1}{J'(\lambda)}\right|=\left|\lambda^{\frac{\alpha}{2}+1}\sqrt{-\frac{2\beta_1}{(2-\alpha)^2}+\frac{2\beta_2}{2-\alpha}-\frac{2\beta_1}{2-\alpha}\log\lambda}\right|\lesssim \lambda^{\frac{\alpha}{2}+1}\sqrt{|\log\lambda|}
\]
and
\[
C_1\tau\leq J(\lambda_{t_1}(\tau))+\xi\left(J(\lambda_{\app}(\tau))-J(\lambda_{t_1}(\tau))\right)\leq C_2\tau,
\]
we obtain
\[
J^{-1}\left(C_1\tau\right)\leq J^{-1}\left(J(\lambda_{t_1}(\tau))+\xi\left(J(\lambda_{\app}(\tau))-J(\lambda_{t_1}(\tau))\right)\right)\leq J^{-1}\left(C_1\tau\right).
\]
Moreover, since
\[
J^{-1}(C\tau)=\lambda_{\app}(C\tau)\approx \left(C\tau\right)^{-\frac{2}{\alpha}}(\log (C\tau))^{-\frac{1}{\alpha}}\approx \tau^{-\frac{2}{\alpha}}(\log \tau)^{-\frac{1}{\alpha}}\approx \lambda_{\app}(\tau)
\]
from Lemma \ref{Jprop}, we obtain
\[
\left|\frac{\tau}{J'(J^{-1}\left(J(\lambda_{t_1}(\tau))+\xi\left(J(\lambda_{\app}(\tau))-J(\lambda_{t_1}(\tau))\right)\right))}\right|\lesssim \tau \lambda_{\app}(\tau)^{\frac{\alpha}{2}+1}\sqrt{|\log\lambda_{\app}(\tau)|}\lesssim \lambda_{\app}(\tau).
\]
Therefore, we obtain
\[
\left|t_{\app}(s)-t_{t_1}(s)\right|\lesssim s^{-\frac{4-\alpha}{\alpha}}(\log s)^{-\frac{2}{\alpha}-1}.
\]
Similarly, we obtain
\[
\left|t_{\app}(s)\right|\approx s^{-\frac{4-\alpha}{\alpha}}(\log s)^{-\frac{2}{\alpha}}.
\]
Moreover, there exists $t_0$ from Lemma \ref{s0s*s'}.
\end{proof}

\begin{corollary}
\label{st}
\[
s_{t_1}(t)\approx |t|^{-\frac{\alpha}{4-\alpha}}|\log|t||^{-\frac{2}{4-\alpha}}
\]
\end{corollary}

\begin{proof}
From Lemma \ref{interval},
\[
\left|t_{\app}(s_{t_1}(t))\right|-C_1s_{t_1}(t)^{-\frac{4-\alpha}{\alpha}}(\log s_{t_1}(t))^{-\frac{2}{\alpha}-1}\leq|t|\leq \left|t_{\app}(s_{t_1}(t))\right|+C_2s_{t_1}(t)^{-\frac{4-\alpha}{\alpha}}(\log s_{t_1}(t))^{-\frac{2}{\alpha}-1}
\]
and
\[
\left|t_{\app}(s_{t_1}(t))\right|\approx s_{t_1}(t)^{-\frac{4-\alpha}{\alpha}}(\log s_{t_1}(t))^{-\frac{2}{\alpha}}
\]
hold. Since
\[
|t|\approx s_{t_1}(t)^{-\frac{4-\alpha}{\alpha}}(\log s_{t_1}(t))^{-\frac{2}{\alpha}},
\]
we obtain
\[
C_1|t|^{-\frac{\alpha}{2}}\leq s_{t_1}(t)^{\frac{4-\alpha}{2}}\log s_{t_1}(t)^{\frac{4-\alpha}{2}}=W_0^{-1}(\log s_{t_1}(t)^{\frac{4-\alpha}{2}})\leq C_2|t|^{-\frac{\alpha}{2}}.
\]
Therefore,
\[
W_0(C_1|t|^{-\frac{\alpha}{2}})\leq \log s_{t_1}(t)^{\frac{4-\alpha}{2}}\leq W_0(C_2|t|^{-\frac{\alpha}{2}}).
\]
Moreover, since $e^{W(z)}=\frac{z}{W(z)}$,
\[
\frac{C_1|t|^{-\frac{\alpha}{2}}}{W_0(C_1|t|^{-\frac{\alpha}{2}})}\leq s_{t_1}(t)^{\frac{4-\alpha}{2}}\leq \frac{C_2|t|^{-\frac{\alpha}{2}}}{W_0(C_2|t|^{-\frac{\alpha}{2}})}.
\]
Since $W_0(z)\approx \log z$ for sufficiently large $z$, we obtain
\[
\frac{C_1|t|^{-\frac{2}{\alpha}}}{C'_1|\log|t||}\leq s_{t_1}(t)^{\frac{4-\alpha}{2}}\leq \frac{C_1|t|^{-\frac{2}{\alpha}}}{C'_2|\log|t||}.
\]
Consequently, we obtain conclusion.
\end{proof}

\begin{lemma}[Conversion of estimates]
\label{uniesti}
For $t\in[t_0,t_1]$, 
\begin{align*}
\tilde{\lambda}_{t_1}(t)&\approx|t|^\frac{2}{4-\alpha}|\log|t||^{\frac{1}{4-\alpha}},& \tilde{b}_{t_1}(t)&\approx |t|^\frac{\alpha}{4-\alpha}|\log|t||^\frac{2}{4-\alpha},\\
\|\tilde{\varepsilon}_{t_1}(t)\|_{H^1}&\lesssim |t|^\frac{\alpha K}{4-\alpha}|\log|t||^{\frac{2K}{4-\alpha}},& \||y|\tilde{\varepsilon}_{t_1}(t)\|_2&\lesssim |t|^\frac{\alpha (K-1)}{4-\alpha}|\log|t||^{\frac{2(K-1)}{4-\alpha}}
\end{align*}
\end{lemma}

\begin{proof}
From Lemma \ref{rebootstrap}, Lemma \ref{lambapprox}, and Corollary \ref{st}, it is proven.
\end{proof}

\section{Proof of Theorem \ref{theorem:EMBS}}
\label{sec:proof}
In this section, we complete the proof of Theorem \ref{theorem:EMBS}. See \cite{LMR,NI} for details of proof.

\begin{proof}[proof of Theorem \ref{theorem:EMBS}]
Let $(t_n)_{n\in\mathbb{N}}\subset(t_0,0)$ be a monotonically increasing sequence such that $\lim_{n\nearrow \infty}t_n=0$. For each $n\in\mathbb{N}$, $u_n$ is the solution for \eqref{NLS} with $\pm=-$ with an initial value
\begin{align*}
u_n(t_n,x):=P_{\lambda_{1,n},b_{1,n},0}(x)
\end{align*}
at $t_n$, where $b_{1,n}$ and $\lambda_{1,n}$ are given by Lemma \ref{paraini} for $t_n$.

According to Lemma \ref{decomposition} with an initial value $\tilde{\gamma}_n(t_n)=0$, there exists a decomposition
\[
u_n(t,x)=\frac{1}{\tilde{\lambda}_n(t)^{\frac{N}{2}}}\left(P+\tilde{\varepsilon}_n\right)\left(t,\frac{x}{\tilde{\lambda}_n(t)}\right)e^{-i\frac{\tilde{b}_n(t)}{4}\frac{|x|^2}{\tilde{\lambda}_n(t)^2}+i\tilde{\gamma}_n(t)}.
\]
Then $(u_n(t_0))_{n\in\mathbb{N}}$ is bounded in $\Sigma^1$. Therefore, up to a subsequence, there exists $u_\infty(t_0)\in \Sigma^1$ such that
\[
u_n(t_0)\rightharpoonup u_\infty(t_0)\quad \mathrm{in}\ \Sigma^1,\quad u_n(t_0)\rightarrow u_\infty(t_0)\quad \mathrm{in}\ L^2(\mathbb{R}^N)\quad (n\rightarrow\infty),
\]
see \cite{LMR,NI} for details.

Let $u_\infty$ be the solution for \eqref{NLS} with $\pm=+$ and an initial value $u_\infty(t_0)$, and let $T^*$ be the supremum of the maximal existence interval of $u_\infty$. Moreover, we define $T:=\min\{0,T^*\}$. Then for any $T'\in[t_0,T)$, $[t_0,T']\subset[t_0,t_n]$ if $n$ is sufficiently large. Then there exist $n_0$ and $C(T',t_0)>0$ such that 
\[
\sup_{n\geq n_0}\|u_n\|_{L^\infty([t_0,T'],\Sigma^1)}\leq C(T',t_0)
\]
holds. Therefore,
\[
u_n\rightarrow u_\infty\quad \mathrm{in}\ C\left([t_0,T'],L^2(\mathbb{R}^N)\right)\quad (n\rightarrow\infty)
\]
holds (see \cite{N}). In particular, $u_n(t)\rightharpoonup u_\infty(t)\ \mathrm{in}\ \Sigma^1$ for any $t\in [t_0,T)$. Furthermore, from the mass conservation, we have
\[
\|u_\infty(t)\|_2=\|u_\infty(t_0)\|_2=\lim_{n\rightarrow\infty}\|u_n(t_0)\|_2=\lim_{n\rightarrow\infty}\|u_n(t_n)\|_2=\lim_{n\rightarrow\infty}\|P(t_n)\|_2=\|Q\|_2.
\]

Based on weak convergence in $H^1(\mathbb{R}^N)$ and Lemma \ref{decomposition}, we decompose $u_\infty$ to
\[
u_\infty(t,x)=\frac{1}{\tilde{\lambda}_\infty(t)^{\frac{N}{2}}}\left(P+\tilde{\varepsilon}_\infty\right)\left(t,\frac{x}{\tilde{\lambda}_\infty(t)}\right)e^{-i\frac{\tilde{b}_\infty(t)}{4}\frac{|x|^2}{{\tilde{\lambda}_\infty(t)}^2}+i\tilde{\gamma}_\infty(t)},
\]
where an initial value of $\tilde{\gamma}_\infty$ is $\gamma_\infty(t_0)\in\left(|t_0|^{-1}-\pi,|t_0|^{-1}+\pi\right]\cap\tilde{\gamma}(u_\infty(t_0))$. Furthermore, for any $t\in[t_0,T)$, as $n\rightarrow\infty$, 
\[
\tilde{\lambda}_n(t)\rightarrow\tilde{\lambda}_\infty(t),\quad \tilde{b}_n(t)\rightarrow \tilde{b}_\infty(t),\quad e^{i\tilde{\gamma}_n(t)}\rightarrow e^{i\tilde{\gamma}_\infty(t)},\quad\tilde{\varepsilon}_n(t)\rightharpoonup \tilde{\varepsilon}_\infty(t)\quad \mathrm{in}\ \Sigma^1
\]
hold. Consequently, from the uniform estimate in Lemma \ref{uniesti}, as $n\rightarrow\infty$, we have
\begin{align*}
\tilde{\lambda}_\infty(t)&\approx|t|^\frac{2}{4-\alpha}|\log|t||^{\frac{1}{4-\alpha}},& \tilde{b}_\infty(t)&\approx |t|^\frac{\alpha}{4-\alpha}|\log|t||^\frac{2}{4-\alpha},\\
\|\tilde{\varepsilon}_\infty(t)\|_{H^1}&\lesssim |t|^\frac{\alpha K}{4-\alpha}|\log|t||^{\frac{2K}{4-\alpha}},& \||y|\tilde{\varepsilon}_\infty(t)\|_2&\lesssim |t|^\frac{\alpha (K-1)}{4-\alpha}|\log|t||^{\frac{2(K-1)}{4-\alpha}}
\end{align*}
Consequently, we obtain that $u$ converges to the blow-up profile in $\Sigma^1$.

Finally, we check energy of $u_\infty$. Since
\[
E\left(u_n\right)-E\left(P_{\tilde{\lambda}_n,\tilde{b}_n,\tilde{\gamma}_n}\right)=\int_0^1\left\langle E'(P_{\tilde{\lambda}_n,\tilde{b}_n,\tilde{\gamma}_n}+\tau \tilde{\varepsilon}_{\tilde{\lambda}_n,\tilde{b}_n,\tilde{\gamma}_n}),\tilde{\varepsilon}_{\tilde{\lambda}_n,\tilde{b}_n,\tilde{\gamma}_n}\right\rangle d\tau
\]
and $E'(w)=-\Delta w-|w|^\frac{4}{N}w+|x|^{-2\sigma}\log|x| w$, we have
\[
E\left(u_n\right)-E\left(P_{\tilde{\lambda}_n,\tilde{b}_n,\tilde{\gamma}_n}\right)=O\left(\frac{1}{{\tilde{\lambda}_n}^2}\|\tilde{\varepsilon}_n\|_{H^1}\right)=O\left(|t|^\frac{\alpha K-4}{4-\alpha}|\log|t||^\frac{2K-2}{4-\alpha}\right).
\]
Similarly, we have
\[
E\left(u_\infty\right)-E\left(P_{\tilde{\lambda}_\infty,\tilde{b}_\infty,\tilde{\gamma}_\infty}\right)=O\left(\frac{1}{{\tilde{\lambda}_\infty}^2}\|\tilde{\varepsilon}_\infty\|_{H^1}\right)=O\left(|t|^\frac{\alpha K-4}{4-\alpha}|\log|t||^\frac{2K-2}{4-\alpha}\right).
\]
From the continuity of $E$, we have
\[
\lim_{n\rightarrow \infty}E\left(P_{\tilde{\lambda}_n,\tilde{b}_n,\tilde{\gamma}_n}\right)=E\left(P_{\tilde{\lambda}_\infty,\tilde{b}_\infty,\tilde{\gamma}_\infty}\right)
\]
and from the conservation of energy,
\[
E\left(u_n\right)=E\left(u_n(t_n)\right)=E\left(P_{\tilde{\lambda}_{1,n},\tilde{b}_{1,n},0}\right)=E_0.
\]
Therefore, we have
\[
E\left(u_\infty\right)=E_0+o_{t\nearrow0}(1)
\]
and since $E\left(u_\infty\right)$ is constant for $t$, $E\left(u_\infty\right)=E_0$.
\end{proof}

\end{document}